\documentclass[11pt, reqno]{amsart}

\usepackage{pstricks} 

\usepackage[all,arc]{xy}
\usepackage{enumerate}
\usepackage{mathrsfs}

\usepackage{amsmath}
\usepackage{amsthm}
\usepackage{amsfonts}
\usepackage{amssymb}
\usepackage{amsbsy}
\usepackage{graphicx}
\usepackage{enumitem}
\usepackage{physics}
\usepackage{hyperref}
\usepackage{bbm}

\newcommand{\mc}[1]{\mathcal{#1}}
\newcommand{\ovl}[1]{\overline{#1}}

\newcommand{\varep}{ \varepsilon }

\newcommand{\sse} {\subseteq}

\newcommand{\Z}{\mathbb{Z}}
\newcommand{\R}{\mathbb{R}}

\newcommand{\E}{\mathbb{E}}

\newcommand{\ra}{\rightarrow}
\newcommand{\toinf}{\ra \infty}

\newcommand{\beq}{\begin{equation}}
\newcommand{\eeq}{\end{equation}}

\newtheorem{theorem}{Theorem}
\newtheorem{prop}[theorem]{Proposition}
\newtheorem{lemma}[theorem]{Lemma}

\theoremstyle{definition}

\newtheorem{remark}[theorem]{Remark}

\newcommand{\p}{\mathbb{P}}
\newcommand{\Var}[1]{\mathrm{Var}(#1)}
\newcommand{\lra}{\longrightarrow}
\numberwithin{equation}{section}
\numberwithin{theorem}{section}
\newcommand{\ind}{\mathbbm{1}}
\newcommand{\ptl}{\partial}

\DeclareMathOperator*{\arginf}{arg\,inf}

\newcommand{\chacorr}{C}
\newcommand{\renyicorr}{R}
\newcommand{\corr}{C}
\newcommand{\pearsoncorr}{\rho}
\newcommand{\cone}{\mc{H}}
\newcommand{\jointdist}{P}
\newcommand{\varemp}[1]{\mathrm{Var}_n(#1)}
\newcommand{\propertyp}{Property P}
\newcommand{\renyi}{R\'{e}nyi}
\newcommand{\moncone}{\mc{M}}
\newcommand{\moncorr}{\chacorr_{mon}}
\newcommand{\convmin}{M}

\newcommand{\dens}{\phi}

\begin{document}

\title[Correlations with tailored extremal properties]{Correlations with tailored extremal properties}
\author{Sky Cao}
\address{\newline Department of Statistics \newline Stanford University\newline Sequoia Hall, 390 Jane Stanford Way \newline Stanford, CA 94305\newline\textup{\tt skycao@stanford.edu}}
\thanks{S.C. was supported by NSF grant DMS RTG 1501767}

\author{Peter J. Bickel}
\address{\newline Department of Statistics \newline University of California, Berkeley \newline Berkeley, California 94720-3860 \newline \textup{\tt bickel@stat.berkeley.edu}}

\keywords{Correlation, shape-restricted regression, isotonic regression.}
\subjclass[2020]{62H20, 62H15}

\begin{abstract}
Recently, Chatterjee has introduced a coefficient of correlation which has several natural properties. In particular, the population version of the coefficient, which generalizes an earlier one of Dette et al., attains its maximal value if and only if one variable is a measurable function of the other variable. In this paper, we seek to define correlations which have a similar property, except now the measurable function must belong to a pre-specified class, which amounts to a shape restriction on the function. We will then look specifically at the correlation corresponding to the class of monotone nondecreasing functions, in which case we can prove various asymptotic results, as well as perform local power calculations. We will also perform local power calculations for Chatterjee's correlation, and for an older one of Dette et al.
\end{abstract}

\maketitle

\section{Introduction}

In a remarkable paper \cite{CH2019}, Sourav Chatterjee proposed a new coefficient of correlation based on an i.i.d. sample $(X_i, Y_i), i = 1, \ldots, n$. Assuming there are no ties among the $X_i$'s and $Y_i$'s (see \cite{CH2019} for the definition in the general case), the correlation is defined as
\beq\label{eq:chacorr-est-def} \hat{\chacorr}_n(X, Y) := 1 - \frac{3 \sum_{i=1}^{n-1} |r_{i+1} - r_i|}{n^2-1}, \eeq
where the $r_i$ are defined as follows. First, sort $X_{(1)} \leq \cdots \leq X_{(n)}$, and for each $i$ let $Y_{(i)}$ be the $Y$ sample corresponding to $X_{(i)}$. Then $r_i$ is defined as the rank of $Y_{(i)}$, i.e. the number of $j$ such that $Y_j \leq Y_{(i)}$. Chatterjee showed that as $n \toinf$, $\hat{\chacorr}_n$ converges a.s. to the population measure
\beq\label{eq:chacorr-def} \chacorr(X, Y) := \frac{\int \Var{\E[\ind(Y \geq t) ~|~ X]} d\mu(t)}{\int \Var{\ind(Y \geq t)} d\mu(t)} = 1 - \frac{\int \E [\mathrm{Var}(\ind(Y \geq t) ~|~ X) ] d\mu(t)}{\int \Var{\ind(Y \geq t)} d\mu(t)},\eeq
where $\mu$ is the law of $Y$. Here $Y$ is assumed to not be constant. In the case where $X, Y$ are continuously distributed, this measure was introduced by Dette et al. \cite{DSS2013} -- see Remark \ref{remark:other-work}. The measure $\chacorr$ has a number of interesting properties:
\begin{enumerate}[label=\Alph*)]
    \item $0 \leq \chacorr \leq 1$.
    \item $\chacorr = 0$ if and only if $X$ and $Y$ are independent.
    \item $\chacorr = 1$ if and only if $Y = h(X)$ a.s. for some measurable function $h : \R \ra \R$.
    \item $\chacorr$ is asymmetric, but can be easily symmetrized to
    \[ \chacorr^*(X, Y) := \max(\chacorr(X, Y), \chacorr(Y, X)),\]
    which clearly satisfies $\chacorr^* = 1$ if and only if $X$ is a function of $Y$ or $Y$ is a function of $X$ (or both).
    \item $\chacorr$ is invariant under strictly increasing transformations of $X$ and $Y$ separately.
\end{enumerate}

This measure is akin to the {\renyi} correlation (also commonly called the maximal correlation), which we shall denote $\renyicorr$ or $\renyicorr(X, Y)$, and is defined as the maximum Pearson correlation between all pairs of $L^2$ functions of $X$ and $Y$ respectively. $R$ may be computed as the square root of the maximal eigenvalue of a compact self adjoint operator,
\[ T : L^2_0(X) \ra L^2_0(X), \]
(or $L^2_0(Y) \ra L^2_0(Y)$ with appropriate changes), where $L^2_0(X)$ is the subspace of $L^2(X)$ consisting of mean zero random variables. The operator $T$ is given by
\[ T(f(X)) := \E[\E[f(X) ~|~ Y] ~|~ X].\]
The {\renyi} correlation is well known to have properties $A, B,$ and $D$, but is symmetric, and 
\begin{enumerate}[label=C*)]
    \item $\renyicorr = 1$ if and only if $g(Y) = h(X)$ for some functions $g$ and $h$, with $g(X) \in L^2(X)$ and $h(Y) \in L^2(Y)$.
\end{enumerate}
An extensive account of the history, computation, and other properties of $\renyicorr$ may be found for instance in \cite{BKRW1993, Renyi1959}.

An advantage of $\chacorr$ is that it gives a clear indication of the functional relationship between $X$ and $Y$ when $\chacorr = 1$. More significantly, an empirical estimate is explicit and simply computable for finite $n$, while unless $X$ and $Y$ are discrete, the empirical version of $\renyicorr$ may only be approximated. On the other hand, $\renyicorr$ is defined if $X$ and $Y$ take values in $\R^p$ and $\R^q$ or more general spaces. Our purposes in this paper are:
\begin{enumerate}
    \item To relate $\chacorr$ more closely to $\renyicorr$.
    \item To extend $\chacorr$ to situations where the $h$ appearing in C) is specified to be monotone or more generally shape restricted.
    \item To study the asymptotic behavior of the sample versions of such measures under independence, when they can be used for testing, as was done by Chatterjee for $\hat{\chacorr}_n$.
\end{enumerate}

\section{Relation between {$\chacorr$} and {$\renyicorr$}}

Note that $\chacorr$ and $\renyicorr$ are unsigned and may be viewed as absolute rather than signed measures of dependence. In fact, we shall argue that $\chacorr$ is closely related not to $\renyicorr$ but to $\renyicorr^2$, the largest eigenvalue of $T$. We begin with the solution to a simpler problem of {\renyi's} or Chatterjee's. For $f_0(Y) \in L^2(Y)$, let
\beq\label{eq:max-corr-def} M(X, f_0(Y)) := \sup_{g(X) \in L^2(X)} \pearsoncorr(f_0(Y), g(X))^2 , \eeq
where $\pearsoncorr$ denotes the Pearson correlation. By Cauchy-Schwarz, or more insightfully the identity, valid if $\Var{f_0(Y)} = \Var{g(X)} = 1$, 
\[ \pearsoncorr(f_0(Y), g(X)) = 1 - \frac{1}{2} \Var{f_0(Y) - g(X)}, \]
a maximizer in \eqref{eq:max-corr-def} is $g(X) = \E[f_0(Y) ~|~ X]$. Hence, 
\beq\label{eq:max-corr-formula} M(X, f_0(Y)) = \frac{\Var{\E [f_0(Y) ~|~ X]}}{\Var{f_0(Y)}} = 1 - \frac{\E \mathrm{Var}(f_0(Y) ~|~ X)}{\Var{f_0(Y)}}. \eeq
If we put $f_0(Y) = Y$ we obtain a measure which satisfies A) and C) but not B) and D). To obtain the last two properties we need only observe that $M = 0$ implies $\E [f_0(Y) ~|~ X] = \E f_0(Y)$. Then to obtain B) and D) note that independence is equivalent to $M(X, \ind(Y > y)) = 0$ for all $y$ in a set $S$ such that $\p(Y \in S) = 1$. This leads fairly naturally to defining $\chacorr$ as in \eqref{eq:chacorr-def}. We note here that this is related to the observation by Azadkia and Chatterjee \cite[Section 3]{AC2019} that $C$ is a mixture of partial $R^2$ statistics. 

There are evidently many ways of constructing such coefficients, but Chatterjee's is particularly elegant since it is easy to see if $F, G$ are the cdfs of $X, Y$ respectively, then $X, Y$ can be replaced by $F(X), G(Y)$, leading naturally to an estimate like $\hat{\chacorr}_n$ which depends on the paired ranks of $(X_i, Y_i)$ (with ties broken at random). Evidently, the empirical estimate of $\chacorr$ requires an estimate of $\E[ \ind(Y \geq t) ~|~ X]$ or $\mathrm{Var}(\ind(Y \geq t) ~|~ X)$. Chatterjee does the latter in a clever way. See also Remark \ref{remark:estimate-conditional-variance}.

\begin{remark}\label{remark:other-work}
\begin{enumerate}
    \item In 2013, Dette et al. \cite{DSS2013} proposed \eqref{eq:chacorr-def} in the case of continuously distributed $X, Y$ as a parameter having properties A)-C) with a different empirical estimate of $C$ than $\hat{C}_n$, which, itself, may be viewed as an estimate of the second expression in \eqref{eq:chacorr-def}. Consistency of the estimate, which they establish, requires smoothness conditions on the conditional distribution and choice of a bandwidth since it requires a growing neighborhood of $X_{(i)}$. Thus, it is not as general.
    \item Independently of our work but at the same time, Shi et al. \cite{SDH2020} made an extensive comparison between $\hat{C}_n$, Dette et al.'s statistic, and other classical tests of the hypothesis of independence which are consistent against all alternatives as is $\hat{C}_n$. They also studied the local power of $\hat{C}_n$, Dette et al.'s statistic, and these other tests against a wide class of contiguous alternatives and showed that $\hat{C}_n$ has no power locally, while the performance of Dette et al.'s may have no power in some cases and yet be rate optimal in others. We had made a more limited comparison and came to similar conclusions in our original posting. However, partly inspired by their results, we show in Section \ref{section:power-calculations} that $\hat{C}_n$ has no power against a wide class of contiguous alternatives. We discuss this and the local power of Dette et al.'s statistic in Sections \ref{section:power-calculations} and \ref{section:semiparametric}. 
    \item The overlap between our work and Shi et al.'s is entirely in Sections \ref{section:power-calculations} and \ref{section:semiparametric}. The results in these sections complement the results of Shi et al. \cite{SDH2020} -- our main focus is giving general conditions under which Chatterjee's statistic and Dette et al.'s statistic have no local power, while the focus of Shi et al. is to exhibit explicit families of alternatives for which the performance of $\hat{C}_n$ and Dette et al.'s statistic is demonstrably worse than other classical tests of independence. 
\end{enumerate}
\end{remark}

The problem we mainly intend to address in this note is how to construct and give appropriate measures for which the value 1 corresponds to a shape restriction on the form of $h(\cdot)$ such that  $Y = h(X)$. Such restrictions are discussed by Guntuboyina and Sen in \cite{GS2018}. They include monotonicity, a special case as we shall see, convexity or concavity, log concavity, and a number of others. They are characterized by the requirement that $h$ belongs to a collection of functions $\cone$, where $\cone$ is such that $\cone_X := \{h(X) : h \in \cone, h(X) \in L^2(X)\}$ is a closed, convex subset in $L^2(X)$. The method we prescribe follows from formula \eqref{eq:max-corr-def}. It is well known that for $\cone$ as above, there exists a nonlinear operator $\Pi_{\cone_X} : L^2(X, Y) \ra \cone_X$ such that for $f(X, Y) \in L^2(X, Y)$,
\[ \Pi_{\cone_X}(f(X, Y)) = \arginf_{h(X) \in \cone_X} \E(f(X, Y) - h(X))^2 .\]
If we substitute $\Pi_{\cone_X}(Y)$ for $f_0(Y)$ or $\E[f_0(Y) ~|~ X]$ in \eqref{eq:max-corr-def}, we evidently get a measure $\chacorr_\cone(X, Y)$ such that $\chacorr_\cone = 1$ if and only if $Y = h(X)$ a.s. with $h(X) \in \cone_X$. If $\cone$ and therefore $\cone_X$ is a convex cone, then it is well known that 
\[ \mathrm{Cov}(Y, \Pi_{\cone_X} (Y)) = \Var{\Pi_{\cone_X}(Y)}.\]
So $\chacorr_\cone$ is also given by \eqref{eq:max-corr-formula} if $\Pi_{\cone_X}(Y)$ replaces $\E[Y ~|~ X]$. That is,
\beq\label{eq:chacorr-cone-def}  \chacorr_\cone := \frac{\Var{\Pi_{\cone_X}(Y)}}{\Var{Y}}.\eeq
Unfortunately, while this measure satisfies properties A), C), and D) above it doesn't satisfy B). However, this can be easily remedied by defining
\beq\label{eq:corr-def} \tilde{\chacorr}_\cone(X, Y) := \frac{1}{2}(\chacorr(X, Y) + \chacorr_\cone(X, Y)),\eeq
which is easily seen to satisfy all of A)-D).

\section{Empirical estimates of $\tilde{\corr}$}

Let $\cone$ be a convex cone of $\ovl{\R}$-valued functions on $\R^d$ (here $\ovl{\R} = \R \cup \{\pm \infty\}$ is the extended real line), containing the constant functions. We now turn to the construction of empirical estimates for $\tilde{\corr}_\cone$. We first consider the general case. By \eqref{eq:corr-def} it is enough to consider $\chacorr_\cone$.  Given a probability measure $\mu$ on $\R^d$, define
\[ \cone_\mu := \cone \cap L^2(\mu) = \{h \in \cone : h \in L^2(\mu)\}. \]
Note $\cone_\mu$ is itself a convex cone. We will assume moreover that for all $\mu$, $\cone_\mu$ is closed in $L^2(\mu)$. (This assumption of closedness is why we need to work with $\ovl{\R}$-valued functions as opposed to $\R$-valued functions.)  For $B > 0$, define
\[ \cone^B := \bigg\{h \in \cone : \sup_{x \in \R^d} |h(x)| \leq B\bigg\}. \]

Let $(X, Y)$ have some joint distribution $\jointdist$, with $X \in \R^d$ and $Y \in \R$. Assume that $Y$ has finite second moment. Let $(X_1, Y_1), \ldots, (X_n, Y_n)$ be i.i.d. samples from $\jointdist$. Let $\mu_X$ be the law of $X$. Let $g \in \cone_{\mu_X}$ be such that
\[ \E(Y - g(X))^2 = \inf_{h \in \cone_{\mu_X}} \E (Y - h(X))^2. \]
By our assumption that $\cone_{\mu_X}$ is a closed convex cone, it follows that $g$ exists and is unique in $L^2(\mu_X)$. Note that $g(X)$ can be interpreted as the projection of $Y$ onto the convex cone $\cone_X  = \{h(X) : h \in \cone_{\mu_X}\} \sse L^2(X)$, so that in the notation of the previous section, we have $\Pi_{\cone_X}(Y) = g(X)$ a.s.

Let $\mu_n$ be the empirical distribution of the $X$ samples, and let $\hat{g}_n$ be such that
\[ \frac{1}{n}\sum_{i=1}^n (Y_i - \hat{g}_n(X_i)^2) = \inf_{h \in \cone_{\mu_n}} \frac{1}{n} \sum_{i=1}^n (Y_i - h(X_i))^2. \]
Again, $\hat{g}_n$ exists and is unique in $L^2(\mu_n)$. Given a function $f(X, Y)$, let
\[ \varemp{f(X, Y)} := \frac{1}{n} \sum_{i=1}^n f(X_i, Y_i)^2 - \bigg(\frac{1}{n} \sum_{i=1}^n f(X_i, Y_i)\bigg)^2, \]
i.e. $\varemp{f(X, Y)}$ is the empirical variance of $f(X, Y)$. Following \eqref{eq:chacorr-cone-def}, define now the empirical correlation
\[ \hat{\corr}_\cone (X, Y) = \hat{\corr}_{\cone, n}(X, Y) := \frac{\varemp{\hat{g}_n(X)}}{\varemp{Y}}. \]
Note that if $\cone$ is the cone of monotone nondecreasing functions, then $\hat{\corr}_\cone$ is the ratio of the empirical variance of the empirical isotonic regression of $Y$ on $X$ to the empirical variance of $Y$.

We now make assumptions on $\cone$ which guarantee convergence of $\hat{\corr}_\cone(X, Y)$ to $\corr_\cone(X, Y)$. The cone $\cone$ is said to have {\propertyp} if in addition to containing the constant functions and $\cone_\mu$ being closed in $L^2(\mu)$ for all probability measures $\mu$, the following two conditions are satisfied for any joint distribution $(X, Y)$ such that $|Y| \leq B$ a.s. for some $B \geq 0$:

\begin{enumerate}
    \item {\bf Boundedness.} We have
    \[ \inf_{h \in \cone^B} \E (Y - h(X))^2 = \inf_{h \in \cone_{\mu_X}} \E(Y - h(X))^2. \]
    Consequently, we may assume that $\sup_{x \in \R^d} |g(x)| \leq B$.
    \item {\bf Glivenko-Cantelli.} We have
    \[ \sup_{h \in \cone^B} \bigg| \frac{1}{n} \sum_{i=1}^n (Y_i - h(X_i))^2 -  \E (Y - h(X))^2 \bigg| \stackrel{a.s.}{\lra} 0.\]
\end{enumerate}

\begin{prop}\label{prop:corr-convergence-property-p}
If $\cone$ satisfies {\propertyp}, then for any $(X, Y)$ where $Y$ is a.s. not a constant and has finite second moment, we have that $\hat{\corr}_\cone(X, Y) \stackrel{a.s.}{\ra} \corr_\cone(X, Y)$.
\end{prop}

We first prove a preliminary result in the case of bounded $Y$.

\begin{lemma}\label{lemma:bounded-case-property-p}
Suppose further that $|Y| \leq B$ a.s. Then 
\[ \frac{1}{n} \sum_{i=1}^n \hat{g}_n(X_i)^2 \stackrel{a.s.}{\lra} \E g(X)^2. \]
\end{lemma}
\begin{proof}
By well known properties of projections onto closed convex cones (see e.g. Lemma 3 of \cite{IM1992}), we have that
\[ \frac{1}{n} \sum_{i=1}^n (Y_i - g(X_i))^2 \geq \frac{1}{n} \sum_{i=1}^n (Y_i - \hat{g}_n(X_i))^2 +  \frac{1}{n} \sum_{i=1}^n (\hat{g}_n(X_i) - g(X_i))^2. \]
By the law of large numbers, we have
\[ \frac{1}{n} \sum_{i=1}^n (Y_i - g(X_i))^2 \stackrel{a.s.}{\lra} \E (Y - g(X))^2. \]
By the boundedness assumption, we may assume that 
\[ \sup_{x \in \R^d} |g(x)|, ~\sup_{x \in \R^d} |\hat{g}_n(x)| \leq B. \]
Thus by the Glivenko-Cantelli assumption and the definition of $g$ as a minimizer, we have
\[ \limsup_{n \toinf}\bigg( \frac{1}{n} \sum_{i=1}^n (Y_i - g(X_i))^2  - \frac{1}{n} \sum_{i=1}^n (Y_i - \hat{g}_n(X_i))^2 \bigg) \stackrel{a.s.}{\leq} 0,\]
from which it follows that
\[ \frac{1}{n} \sum_{i=1}^n (\hat{g}_n(X_i) - g(X_i))^2 \stackrel{a.s.}{\lra} 0, \]
and thus
\[ \frac{1}{n} \sum_{i=1}^n \hat{g}_n(X_i)^2 \stackrel{a.s.}{\lra} \E g(X)^2, \]
as desired.
\end{proof}

\begin{proof}[Proof of Proposition \ref{prop:corr-convergence-property-p}]
First, observe that by part (iv) of Lemma 3 of \cite{IM1992}, and the fact that $\cone$ contains the constant functions, we have
\[ \E(Y - g(X)) \leq 0, ~~ \E (Y - g(X)) (-1) \leq 0, \]
which implies $\E g(X) = \E Y$. Similarly, in the empirical case, we have
\[ \frac{1}{n} \sum_{i=1}^n \hat{g}_n(X_i) = \frac{1}{n} \sum_{i=1}^n Y_i. \]
Thus by the law of large numbers, we obtain convergence of the sample mean of $\hat{g}_n$ to $\E g(X)$. 

Now onto the second moments. For $B > 0$, define the function
\[ \varphi_B(x) := \min(\max(x, -B), B). \]
Using the boundedness assumption of {\propertyp}, let $g^B, \hat{g}^B_n \in \cone^B$ be such that
\[ \E (\varphi_B(Y) - g^B(X))^2 = \inf_{h \in \cone_{\mu_X}} \E (Y - h(X))^2, \]
\[ \frac{1}{n} \sum_{i=1}^n (\varphi_B(Y_i) - \hat{g}^B_n(X_i))^2 = \inf_{h \in \cone_{\mu_n}} \frac{1}{n} \sum_{i=1}^n (\varphi_B(Y_i) - h(X_i))^2. \]
By well known properties of projection onto closed convex cones (see e.g. Lemma \cite{IM1992}), we have
\beq\label{eq:second-moment-inner-product-identity} \E Y g(X) = \E g(X)^2, ~~\frac{1}{n} \sum_{i=1}^n Y_i \hat{g}_n(X_i) = \frac{1}{n} \sum_{i=1}^n \hat{g}_n(X_i)^2. \eeq
Thus it suffices to show
\[ \frac{1}{n} \sum_{i=1}^n Y_i \hat{g}_n(X_i)\stackrel{a.s.}{\lra} \E Yg(X). \]
To start, observe for any $B > 0$,
\[ \begin{split}
\bigg|\E Y g(X) - \frac{1}{n} \sum_{i=1}^n Y_i \hat{g}_n(X_i)\bigg| \leq& ~|\E Yg(X) - \E \varphi_B(Y) g^B(X)| ~+ \\
&\bigg| \E \varphi_B(Y) g^B(X) - \frac{1}{n} \sum_{i=1}^n \varphi_B(Y_i) \hat{g}_n^B(X_i) \bigg| ~+\\
& \bigg| \frac{1}{n} \sum_{i=1}^n \varphi_B(Y_i) \hat{g}_n^B(X_i) - \frac{1}{n} \sum_{i=1}^n Y_i \hat{g}_n(X_i)\bigg| .
\end{split}\]
By Lemma \ref{lemma:bounded-case-property-p}, and an observation analogous to \eqref{eq:second-moment-inner-product-identity}, the middle term in the right hand side above converges a.s. to 0. Thus it suffices to show that there is some function $\delta : \R_{\geq 0} \ra \R_{\geq 0}$, such that $\lim_{B \toinf} \delta(B) = 0$, and for all $B > 0$,
\[ |\E Yg(X) - \E \varphi_B(Y) g^B(X)| \leq \delta(B), \]
\[ \limsup_{n \toinf}\bigg| \frac{1}{n} \sum_{i=1}^n \varphi_B(Y_i) \hat{g}_n^B(X_i) - \frac{1}{n} \sum_{i=1}^n Y_i \hat{g}_n(X_i)\bigg| \stackrel{a.s.}{\leq} \delta(B). \]
To see this, observe
\[ \begin{split}
|\E Y g(X) - \E \varphi_B(Y) g^B(X)| \leq |\E (Y - &\varphi_B(Y)) g(X)| ~+  \\
&|\E \varphi_B(Y) (g(X) - g^B(X))| .
\end{split}\]
By Cauchy-Schwarz and the fact that projections are contractions, we have that the right hand side above may be bounded by
\[ (\E Y^2 \ind(|Y| > B))^{1/2} (\E Y^2)^{1/2} + (\E Y^2)^{1/2} (\E (g(X) - g^B(X))^2)^{1/2} .\]
Now since the projection map is 1-Lipschitz (see e.g. part (vi) of Lemma 3 of \cite{IM1992}), we have
\[(\E (g(X) - g^B(X))^2)^{1/2} \leq (\E (Y - \varphi_B(Y))^2)^{1/2} \leq (\E Y^2 \ind(|Y| > B))^{1/2}.\]
Combining the previous few displays, we thus obtain
\[ |\E Y g(X) - \E \varphi_B(Y) g^B(X)| \leq 2 (\E Y^2 \ind(|Y| > B))^{1/2} (\E Y^2)^{1/2}. \]
Applying the same argument to the sample quantities, we may similarly obtain
\[ \begin{split}
\bigg| \frac{1}{n} \sum_{i=1}^n \varphi_B(Y_i) \hat{g}_n^B(X_i) - &\frac{1}{n} \sum_{i=1}^n Y_i \hat{g}_n(X_i)\bigg| \leq  \\
&2 \bigg(\frac{1}{n} \sum_{i=1}^n Y_i^2 \ind(|Y_i| > B)\bigg)^{1/2} \bigg(\frac{1}{n} \sum_{i=1}^n Y_i^2\bigg)^{1/2}.
\end{split}\]
We thus see that the function
\[ \delta(B) := 2 (\E Y^2 \ind(|Y| > B))^{1/2} (\E Y^2)^{1/2}\]
has the desired properties, and thus the desired result now follows.
\end{proof}

\begin{remark}\label{remark:estimate-conditional-variance}
This simple approach does not work when $\cone$ is the set of all measurable functions (so that $\cone_{\mu_X} = L^2(\mu_X)$), unless $X$ is discrete, since then the empirical projection $\hat{g}_n$ will perfectly match the $Y$ samples, so that $\hat{\chacorr}_\cone$ will always be 1. Chatterjee's approach is essentially to use $\mathrm{Var}(Y ~|~ X) = \E[(Y - Y')^2 ~|~ X] / 2$, where given $X$, the pair $(Y, Y')$ is i.i.d. Empirically, no such pair of $Y$'s is available, but if $X = X_{(i)}$ the $i$th order statistic, he essentially approximates by using $(Y_{(i)}, Y_{(i+1)})$. However, since the second identity in \eqref{eq:chacorr-def} no longer holds, this seems fruitless for $\cone \neq L^2$.
\end{remark}

\section{The isotonic case}

The one dimensional isotonic case, where $\moncone$ is the set of monotone nondecreasing functions $\R \ra \ovl{\R}$, is special in a number of ways as noted by many authors. For one, by a small trick, we can define the correlation not just for $Y$ with finite second moment, but actually for general $Y$, and the empirical version of this correlation will also satisfy property E). Secondly, it is the only case so far where we are able to verify {\propertyp} and thus show convergence of $\hat{\chacorr}_\moncone$, and also establish the behavior of the statistic under independence. However, we mention here as a side note that it could be possible that if we choose $\cone$ to be a small convex subset of a cone, not necessarily closed (for instance, $\cone = \{g \in L^2(\mu) : g \text{ is convex and 1-Lipschitz}\}$), and assume $\E[Y ~|~ X] \in \cone$, then we may obtain {\propertyp} from known results \cite{GS2018}. In addition, we shall provide an alternative based on Spearman's correlation which is simpler to analyze and also has properties A)-E) (where property C) is suitably modified). 
We first verify that $\moncone$ indeed satisfies {\propertyp}.

\begin{lemma}
The cone $\mc{M}$ satisfies {\propertyp}.
\end{lemma}
\begin{proof}
Clearly $\moncone$ contains the constant functions. Let $\mu$ be a probability measure on $\R$. To see why $\moncone_\mu = \{ h : h \in L^2(\mu)\}$ is closed in $L^2(\mu)$, suppose we have a sequence $\{g_n\}_{n \geq 1}$, such that $g_n \ra g \in L^2(\mu)$. We may then extract a subsequence $g_{n_k}$ which converges to $g$ $\mu$-a.s. Thus if we define $\tilde{g} := \limsup_k g_{n_k}$, we have that $\tilde{g}$ is nondecreasing, and also $\tilde{g} = g$ $\mu$-a.e. Note here is why we need to work with $\ovl{\R}$-valued functions, since even if $g_{n_k}$ is $\R$-valued for all $k$, it could be that $\tilde{g}$ is $\ovl{\R}$-valued.

The boundedness property is clearly satisfied by $\moncone$. Finally, to show the Glivenko-Cantelli property, fix $B > 0$ and observe by e.g. \cite[Example 19.11]{VDV1998} combined with \cite[Theorem 19.4]{VDV1998}, we have that both
\[ \sup_{h \in \moncone^B} \bigg|\frac{1}{n}\sum_{i=1}^n h(X_i) - \E h(X_i)\bigg|, \sup_{h \in \moncone^B} \bigg|\frac{1}{n}\sum_{i=1}^n h(X_i)^2 - \E h(X_i)^2\bigg|  \stackrel{a.s.}{\lra} 0. \]
Given $|Y| \leq B$, we will thus be done if we can show
\beq\label{eq:cross-term-conv} \sup_{h \in \moncone^B} \bigg| \frac{1}{n} \sum_{i=1}^n Y_i h(X_i) - \E Y h(X)\bigg| \stackrel{a.s.}{\lra} 0.\eeq
Towards this end, note
\[ \frac{1}{n}\sum_{i=1}^n Y_i h(X_i) = \frac{1}{n} \sum_{i=1}^n (Y_i)_+ h(X_i) - \frac{1}{n} \sum_{i=1}^n (Y_i)_{-} h(X_i) ,\]
where $(Y_i)_+ = \max(Y_i, 0)$ and $(Y_i)_{-} = -\min(Y_i, 0)$. A similar identity holds for the population version, and thus by the triangle inequality, it suffices to show \eqref{eq:cross-term-conv} under the further assumption $0 \leq Y \leq B$. But now that $Y$ is non-negative, we may apply a standard bracketing argument (see e.g. the proof of \cite[Theorem 2.4.1]{VDVW1996}), using the fact that the bracketing numbers for $\moncone^B$ are finite (see e.g. \cite[Theorem 2.7.5]{VDVW1996}).
\end{proof}

\subsection{The case of general $Y$}\label{section:monotone-general-y}

To define a correlation for general $Y$, the key observation is the following. Let $G$ be the cdf of $Y$. Define $G^{-1} : \R \ra \R$ by
\[ G^{-1}(x) := \inf \{t : G(t) \geq x\}. \]

\begin{lemma}\label{lemma:y-equals-g-inverse-g}
We have $Y \stackrel{a.s.}{=} G^{-1}(G(Y))$. 
\end{lemma}
\begin{remark}
Observe that both $G, G^{-1}$ are nondecreasing. Thus this lemma shows that $Y$ is a nondecreasing function of $X$ if and only if $G(Y)$ is a nondecreasing function of $X$. This will mean that Property C) holds for the correlation we will soon define.
\end{remark}
\begin{proof}
Observe by definition that for all $x \in \R$, we have $x \geq G^{-1}(G(x))$. Let $\mc{C} := \{x \in \R : x > G^{-1}(G(x))\}$. 
It suffices to show that $\p(Y \in \mc{C}) = 0$. Given $x \in \R$, define $a_x := G^{-1}(G(x))$, and $b_x := \sup \{x' : G(x') = G(x)\}$. If $x \in \mc{C}$, then $a_x < x \leq b_x$. I now claim that for any $x, x' \in \mc{C}$, the intervals $(a_x, b_x)$, $(a_{x'}, b_{x'})$ are either disjoint or the same. To see this, first note if $a_x = a_{x'}$, then $G^{-1}(G(x)) = G^{-1}(G(x'))$. As $G^{-1}(G(x)) \leq x$, we have $G(G^{-1}(G(x))) \leq G(x)$. Moreover, we have
\[ G(G^{-1}(G(x))) = G(G^{-1}(G(x'))) \geq G(x'). \]
The same argument with $x, x'$ switched then gives $G(x) = G(x')$, and thus we see that $b_x = b_{x'}$.

Now suppose $a_x < a_{x'}$. We now show that $b_x \leq a_{x'}$. Suppose $\tilde{x}$ is such that $G(\tilde{x}) = G(x)$. Then by assumption, $G^{-1}(G(\tilde{x})) < G^{-1}(G(x'))$. This implies $G(\tilde{x}) < G(x')$, which implies $\tilde{x} < G^{-1}(G(x'))$. Taking supremum over all such $\tilde{x}$, we obtain the desired inequality.

We thus have that $\{(a_x, b_x) : x \in \mc{C}\}$ is a countable collection of intervals, and thus also $\mc{D} := \{x \in \mc{C} : x = b_x\}$ is countable. Note also that for any $x \in \mc{C}$, we have $\p(Y \in (a_x, x]) = 0$, which implies that $\p(Y \in (a_x, b_x)) = 0$, and if additionally $x = b_x$, then $\p(Y = x) = 0$. As
\[ \mc{C} \sse \mc{D} \cup  \bigcup_{x \in \mc{C}} (a_x, b_x) , \]
we may thus conclude by a union bound that $\p(Y \in \mc{C}) = 0$, as desired.
\end{proof}

Let us now define a population correlation $\moncorr$ as follows:
\[ \moncorr(X, Y) := \chacorr_\moncone(X, G(Y)). \]
When defining the empirical correlation, there is this problem that we may not know $G$. Instead, we can plug in the estimate $\hat{G}_n$, which is the empirical cdf of the $Y$ samples. Thus, define the empirical correlation as follows:
\[ \hat{\chacorr}_{mon}(X, Y) = \hat{\chacorr}_{mon, n}(X, Y) := \hat{\chacorr}_{\moncone, n}(X, \hat{G}_n(Y)).\]
In other words, we are simply replacing $Y_i$ by its rank $\hat{G}_n(Y_i)$, for each $1 \leq i \leq n$.

In practice, the empirical correlation may be computed as follows. First, sort $X_{(1)} \leq \cdots \leq X_{(n)}$. For each $1 \leq i \leq n$, let $Y_{(i)}$ be the $Y$ sample corresponding to $X_{(i)}$. Let $\hat{z}_1 \leq \cdots \leq \hat{z}_n$ be the solution to the isotonic regression
\[ \inf_{z_1 \leq \cdots \leq z_n} \frac{1}{n} \sum_{i=1}^n(\hat{G}_n(Y_{(i)}) - z_i)^2, \]
where there is the restriction that if $X_{(i)} = X_{(i+1)}$, then $z_i = z_{i+1}$. Then
\[ \hat{\chacorr}_{mon}(X, Y) = \frac{\frac{1}{n} \sum_{i=1}^n \hat{z}_i^2 - \bigg(\frac{1}{n} \sum_{i=1}^n \hat{G}_n(Y_i)\bigg)^2}{\mathrm{Var}_n(\hat{G}_n(Y))}.\]
Here we have also used the fact that
\[ \frac{1}{n}\sum_{i=1}^n \hat{z}_i = \frac{1}{n} \sum_{i=1}^n \hat{G}_n(Y_i). \]
Sorting the $X$ sample takes time $O(n \log n)$, the isotonic regression can be done in time $O(n)$ by the pool adjacent violators algorithm, and all other computations take time $O(n)$. Thus the empirical correlation may be calculated in time $O(n \log n)$.

\begin{prop}\label{prop:moncorr-conv}
For any $(X, Y)$ where $Y$ is a.s. not a constant, we have that
\[ \hat{\chacorr}_{mon}(X, Y) \stackrel{a.s.}{\lra} \moncorr(X, Y).\]
\end{prop}
\begin{proof}
Note since $\moncone$ satisfies {\propertyp} and $0 \leq \hat{G}_n, G \leq 1$, by the boundedness assumption it suffices to just optimize over $\moncone^1$. We have by the Glivenko-Cantelli assumption that
\[ \sup_{h \in \moncone^1} \bigg| \frac{1}{n} \sum_{i=1}^n (G(Y_i) - h(X_i))^2 - \E(G(Y) - h(X))^2\bigg| \stackrel{a.s.}{\lra} 0. \]
Now for any $h \in \moncone^1$, we have
\[ \bigg| \frac{1}{n} \sum_{i=1}^n (G(Y_i) - h(X_i))^2 - \frac{1}{n} \sum_{i=1}^n (\hat{G}_n(Y_i) - h(X_i))^2 \bigg| \leq 4 \sup_{x \in \R} |\hat{G}_n(x) - G(x)|. \]
By the Glivenko-Cantelli theorem, we have
\[ \sup_{x \in \R}  |\hat{G}_n(x) - G(x)|\stackrel{a.s.}{\lra} 0 , \]
and thus we obtain
\[ \sup_{h \in \moncone^1} \bigg| \frac{1}{n} \sum_{i=1}^n (\hat{G}_n(Y_i) - h(X_i))^2 - \E(G(Y) - h(X))^2\bigg| \stackrel{a.s.}{\lra} 0. \]
The rest of the proof now proceeds as in the proofs of Lemma \ref{lemma:bounded-case-property-p} and Proposition \ref{prop:corr-convergence-property-p}.
\end{proof}

Similar to \eqref{eq:corr-def}, we now define the averaged (population) correlation 
\[ \tilde{\chacorr}_{mon}(X, Y) := \frac{1}{2} (\chacorr(X, Y)+ \moncorr^{1/2}(X, Y)), \]
as well as the empirical version
\[ \hat{\tilde{\chacorr}}_{mon, n}(X, Y) := \frac{1}{2} (\hat{\chacorr}_n(X, Y) + \hat{\chacorr}^{1/2}_{mon}(X, Y)).\]
Here we use $\moncorr^{1/2}$ rather than $\moncorr$, because as we shall see, the asymptotic theory under independence is nicer. The population version satisfies properties A)-E) (with property C) suitably adjusted), and is defined for general (non-constant) $Y$. By Proposition \ref{prop:moncorr-conv}, we have that the empirical correlation converges a.s. to the population correlation. 

\subsection{Asymptotic behavior under independence and continuity}\label{section:moncorr-asymptotics}

We next investigate the asymptotic distribution of $\hat{\tilde{\chacorr}}_{mon}(X, Y)$ under the assumptions that $X, Y$ are independent and have continuous distributions. 

\begin{theorem}\label{thm:averaged-correlation-asymptotics}
Assume $X, Y$ are independent and continuously distributed. 
Then
\[ \sqrt{n} \bigg(\hat{\tilde{\chacorr}}_{mon}(X, Y) - \frac{1}{2}\sqrt{\frac{\log n}{n}}\bigg) \stackrel{d}{\lra} N(0, 23/80).\] 
\end{theorem}

This theorem follows from the following proposition about the joint asymptotics of the two empirical correlations.

\begin{prop}\label{prop:correlation-joint-asymptotics}
Assume $X, Y$ are independent and continuously distributed. Then
\[ \bigg(\sqrt{n} \hat{\chacorr}_n(X, Y), \frac{n \hat{\chacorr}_{mon}(X, Y) - \log n}{\sqrt{\log n}}\bigg) \stackrel{d}{\lra} N\bigg(\begin{pmatrix} 0 \\ 0\end{pmatrix}, \begin{pmatrix} 2/5 & 0 \\ 0 & 3\end{pmatrix} \bigg). \]
\end{prop}

\begin{proof}[Proof of Theorem \ref{thm:averaged-correlation-asymptotics}]
Let
\[ Z_n := n \hat{\chacorr}_{mon}(X, Y). \]
By Taylor's remainder theorem, we have
\[ \sqrt{Z_n} = \sqrt{\log n} + \frac{1}{2 \sqrt{\log n}} (Z_n - \log n) - \frac{1}{8\xi_n^{3/2}} (Z_n - \log n)^2,\]
where $\xi_n$ is between $Z_n$ and $\log n$. By Proposition \ref{prop:correlation-joint-asymptotics}, we obtain
\[ \sqrt{Z_n} - \sqrt{\log n} = \frac{1}{2 \sqrt{\log n}}(Z_n - \log n) + o_P(1).  \]
We thus have
\[ \begin{split}
\sqrt{n}\bigg(\hat{\tilde{\chacorr}}_{mon}&(X, Y) - \frac{1}{2}\sqrt{\frac{\log n}{n}}\bigg) = \\
&\frac{1}{2}\sqrt{n} \hat{\chacorr}_n(X, Y) + \frac{1}{4} \frac{n \hat{\chacorr}_{mon}(X, Y) - \log n}{\sqrt{\log n}} + o_P(1)  .
\end{split}\]
We may now finish by Proposition \ref{prop:correlation-joint-asymptotics} and Slutsky's lemma.
\end{proof}

\begin{remark}
As we will see, the distribution of $n \hat{\chacorr}_{mon}(X, Y)$ may be approximately described as follows. First, sample $N_n$, which is distributed as the number of cycles in a uniform random permutation on $[n]$ (and which has an explicit distributional representation as a sum of independent Bernoullis; see e.g. \cite[Section 2]{STEELE2002}). Then generate a $\chi^2_{(N_n)}$. This explains the scaling in the central limit theorem, since it is known that $(N_n - \log n) / \sqrt{\log n} \stackrel{d}{\lra} N(0, 1)$ (see e.g. \cite[Section 2]{STEELE2002}), and $(\chi^2_n - n) / \sqrt{n} \stackrel{d}{\lra} N(0, 2)$. 
\end{remark}

We now begin towards the proof of Proposition \ref{prop:correlation-joint-asymptotics}. To start, sort $X_{(1)} < \cdots < X_{(n)}$ (the inequalities are strict since $X$ has continuous distribution), and let $Y_{(i)}$ be the $Y$ sample corresponding to $X_{(i)}$. Let $\hat{z}_1 \leq \cdots \leq \hat{z}_n$ be the isotonic regression of $\hat{G}_n(Y)$ on $X$, i.e. the solution to the minimization problem
\[ \inf_{z_1 \leq \cdots \leq z_n} \frac{1}{n} \sum_{i=1}^n (\hat{G}_n(Y_{(i)}) - z_i)^2. \]
The solution satisfies
\[ \frac{1}{n} \sum_{i=1}^n \hat{z}_i = \frac{1}{n} \sum_{i=1}^n \hat{G}_n(Y_i). \]
Since $Y$ is continuous, the right hand side above will deterministically be 
\[ \mu_n := \frac{1}{2} \bigg(1 + \frac{1}{n}\bigg),\]
and moreover
\[\sigma_n^2 := \frac{1}{12} \bigg(1 - \frac{1}{n^2}\bigg) =  \mathrm{Var}_n(\hat{G}_n(Y)). \]
We then have that
\[\hat{\chacorr}_{mon}(X, Y) = \frac{1}{\sigma_n^2} \frac{1}{n} \sum_{i=1}^n (\hat{z}_i - \mu_n)^2 = \frac{1}{n} \sum_{i=1}^n \bigg(\frac{\hat{z}_i - \mu_n}{\sigma_n}\bigg)^2.  \]
Note also that if $\tilde{z}_1, \ldots, \tilde{z}_n$ is the isotonic regression of $(\hat{G}_n(Y) - \mu_n) / \sigma_n$ on $X$, i.e. the solution to the minimization problem
\[ \inf_{z_1 \leq \cdots \leq z_n} \frac{1}{n} \sum_{i=1}^n \bigg(\frac{\hat{G}_n(Y_{(i)}) - \mu_n}{\sigma_n} - z_i\bigg)^2, \]
then we have that $(\hat{z}_i - \mu_n) / \sigma_n = \tilde{z}_i$ for all $1 \leq i \leq n$, and thus
\[ \hat{\chacorr}_{mon}(X, Y) = \frac{1}{n} \sum_{i=1}^n \tilde{z}_i^2.\]

Since $X, Y$ are independent and $Y$ is continuous, we have that the random vector $(n\hat{G}_n(Y_1), \ldots, n\hat{G}_n(Y_n))$ has the same distribution as $\pi$, a uniform random permutation on $[n] := \{1, \ldots, n\}$. We thus have that $(\tilde{z}_1, \ldots, \tilde{z}_n)$ has the same distribution as $(\hat{w}_1, \ldots, \hat{w}_n)$, where the latter is the isotonic regression of $(\pi / n - \mu_n) / \sigma$, i.e. solution to the minimization problem
\[ \inf_{w_1 \leq \cdots \leq w_n} \frac{1}{n} \sum_{i=1}^n \bigg(\frac{\pi(i) / n - \mu_n}{\sigma_n} - w_i\bigg)^2. \]
Recalling the definition \eqref{eq:chacorr-est-def} of $\hat{\chacorr}_n(X, Y)$, we further obtain that the random vector $(\hat{\chacorr}_n(X, Y), \hat{\chacorr}_{mon}(X, Y))$ has the same distribution as
\beq\label{eq:function-of-permutation} \Bigg(1 - \frac{3 \sum_{i=1}^{n-1} |\pi(i) - \pi(i+1)|}{n^2-1}, ~\frac{1}{n} \sum_{i=1}^n \hat{w}_i^2 \Bigg) . \eeq
We thus have reduced the problem to analyzing statistics of a uniform random permutation.

The key tool in our analysis will be the following bijection on permutations, which we now begin to describe, following \cite[Section 3]{STEELE2002}. We start by fixing some real numbers $y_1, \ldots, y_n$ which are linearly independent over $\Z$; i.e., if $a_1, \ldots, a_n \in \Z$ are such that $a_1y_1 + \cdots + a_n y_n= 0$, then $a_1 = \cdots = a_n = 0$. Given a permutation $\tau = (\tau(1), \ldots, \tau(n))$, define the cumulative sum process $S_\tau : [0, n] \ra \R$ by linearly interpolating between the points $S_\tau(0) := 0$, $S_\tau(i) := y_{\tau(1)} + \cdots + y_{\tau(i)}$ for $1 \leq i \leq n$. Let $\convmin_\tau$ be the greatest convex minorant of $S_\tau$ (technically \cite[Section 3]{STEELE2002} considers the least concave majorant, but by a sign change we see that everything in the section also applies to the greatest convex minorant). Note $\convmin_\tau$ will be a piecewise linear function, and so let $0 = i_0 < i_1 < \cdots < i_m = n$ denote the knots of $\convmin_\tau$. Now define the permutation $\tilde{\tau}$ as the product of cycles
\beq\label{eq:bohnenblust-spitzer-def} \tilde{\tau} := (\tau_{i_0 + 1}, \ldots, \tau_{i_1}) (\tau_{i_1 + 1}, \ldots, \tau_{i_2}) \cdots (\tau_{i_{m-1} + 1}, \ldots, \tau_{i_m}). \eeq
It is proven in \cite[Section 4]{STEELE2002} that this map $\tau \mapsto \tilde{\tau}$ is a bijection. Call this map $B$. To be clear, $B$ is defined by the real numbers $y_1, \ldots, y_n$, which were assumed to be linearly independent. This bijection on permutations is called the Bohnenblust-Spitzer algorithm.

To see how the Bohnenblust-Spitzer algorithm relates to our current situation, we now describe a well known explicit representation for the $\hat{w}_i$. For $1 \leq i \leq n$, define
\[ x_i := \frac{1}{\sigma_n} \bigg(\frac{i}{n} - \mu_n\bigg). \]
Let $S_\pi$ be the cumulative sum process defined as in the previous paragraph, but now using $x_1, \ldots, x_n$. Let $\convmin_\pi$ be the greatest convex minorant of $S_\pi$. Then for each $1 \leq i \leq n$, $\hat{w}_i$ is equal to the the left hand slope of $\convmin_\pi$ at $i$ (see e.g. \cite[Theorem 1.2.1]{RWD1988}). So given two consecutive knots $i_{k-1} < i_k$ of $\convmin_\pi$, and a point $i_{k-1} < i \leq i_k$, we have
\beq\label{eq:graphical-rep} \hat{w}_i = \frac{S_\pi(i_k) - S_\pi(i_{k-1})}{i_k - i_{k-1}} = \frac{x_{\pi(i_{k-1} + 1)} + \cdots + x_{\pi(i_k)}}{i_k - i_{k-1}}. \eeq

Now define the functions on permutations of $[n]$
\[ f_1(\tau) := 1 - \frac{3n}{n^2-1}\sum_{C \in \tau} \sum_{(i, j) \in C} |i/n - j/n|, \]
\[ f_2(\tau) := \sum_{C \in \tau} \bigg(\frac{1}{\sqrt{|C|}} \sum_{i \in C} x_i\bigg)^2, \]
where $\sum_{C \in \tau}$ denotes summation over the cycles of $\tau$, $|C|$ is the length of $C$, and if $C = (i_1, \ldots, i_k)$, then $\sum_{(i, j) \in C}$ means we are summing over consecutive pairs $(i_1, i_2), (i_2, i_3), \ldots, (i_k, i_1)$, and $\sum_{i \in C}$ means we are summing over $i_1, \ldots, i_k$. The next lemma allows us to prove Proposition \ref{prop:correlation-joint-asymptotics} by studying $(f_1(\rho), f_2(\rho))$, for $\rho$ a uniform random permutation on $[n]$.

\begin{lemma}\label{lemma:bohnenblust-spitzer-application}
For all $n \geq 2$, there is a coupling $(\pi, \rho)$, such that both $\pi, \rho$ are uniform random permutations on $[n]$, and 
\[ \sqrt{n} \bigg(1 - \frac{3 \sum_{i=1}^{n-1} |\pi(i) - \pi(i+1)|}{n^2-1}\bigg) = \sqrt{n} f_1(\rho) + o_P(1),  \]
\[ \sum_{i=1}^n \hat{w}_i^2 = f_2(\rho) + o_P(1).\]
\end{lemma}
\begin{proof}
Fix $n \geq 2$. We use the Bohnenblust-Spitzer algorithm. There is the slight problem that $x_1, \ldots, x_n$ are not linearly independent over $\Z$. This can be remedied by introducing $\delta := 2^{-n}$ (say), and taking a perturbation $x_1^\delta, \ldots, x_n^\delta$ which is linearly independent over $\Z$ and such that $|x_i^\delta - x_i| \leq \delta$ for all $1 \leq i \leq n$. Let $B^\delta$ be the bijection given by the Bohnenblust-Spitzer algorithm applied with $x_1^\delta, \ldots, x_n^\delta$. Since $B^\delta$ is a bijection, $B^\delta(\pi)$ is also a uniform random permutation. Set $\rho := B^\delta(\pi)$. 

We show the second statement first. Let $\hat{w}^\delta_1, \ldots, \hat{w}^\delta_n$ be the isotonic regression of $x_{\pi(1)}^\delta, \ldots, x_{\pi(n)}^\delta$. First, since isotonic regression is a projection onto a convex cone and thus is 1-Lipschitz, we have
\[ \sum_{i=1}^n (\hat{w}_i - \hat{w}^\delta_i)^2 \leq \sum_{i=1}^n (x_i^\delta - x_i)^2 \leq n \delta^2,\]
and thus applying Cauchy-Schwarz, we obtain
\[ \bigg|\sum_{i=1}^n \hat{w}_i^2 - \sum_{i=1}^n (\hat{w}^\delta_i)^2\bigg| = O(n \delta).\]
Now observe that by the definition of $B^\delta(\pi)$ \eqref{eq:bohnenblust-spitzer-def} and the characterization of the isotonic regression \eqref{eq:graphical-rep}, we have
\[ \sum_{i=1}^n (\hat{w}^\delta_i)^2 = \sum_{C \in \rho} |C| \bigg(\frac{1}{|C|} \sum_{i \in C} x^\delta_i\bigg)^2 = \sum_{C \in \rho} \bigg(\frac{1}{\sqrt{|C|}} \sum_{i \in C} x^\delta_i\bigg)^2 =: f_2^\delta(\rho),\]
and moreover
\[|f_2^\delta(\rho) - f_2(\rho)| = \sum_{C \in \rho} O(\sqrt{|C|}) \sqrt{|C|}\delta = O(n \delta).  \]
Putting everything together, we obtain the second statement (actually we have proven something slightly stronger -- the $o_P(1)$ can be replaced by $o(1)$).

For the first statement, observe that the differences only arise when we wrap around a cycle of $\rho$, or at the boundary between two cycles of $\rho$. Let $N_n$ be the number of cycles of $\rho$. This then gives
\[ \bigg| \sqrt{n} \bigg(1 - \frac{3 \sum_{i=1}^{n-1} |\pi(i) - \pi(i+1)|}{n^2-1} - f_1(\rho)\bigg)\bigg| 
= O\bigg(\frac{N_n}{\sqrt{n}}\bigg). \]
The desired statement now follows since $\E N_n = O(\log n)$ (see e.g. \cite[Section 2]{STEELE2002}).
\end{proof}

We have thus reduced to studying the joint asymptotics of $(f_1(\rho), f_2(\rho))$, for a uniform random permutation $\rho$. We will now suppose $\rho$ is sampled as follows. First, sample $L_1 \geq \cdots \geq L_{N_n}$, which are distributed as the ranked cycle lengths of a uniform random permutation on $[n]$ (and so $N_n$ is distributed as the number of cycles). For $1 \leq i \leq N_n$, let $a_i := L_1 + \cdots + L_i$, and let $a_0 := 0$. Independently of $L_1, \ldots, L_{N_n}$, sample $V_1, \ldots, V_n \stackrel{i.i.d.}{\sim} \mathrm{Unif}(0, 1)$, and let $\hat{F}_n$ be the empirical cdf of the $V$ sample. Let $\eta$ be the permutation defined by $\eta(i) := \hat{F}_n(V_i)$, $1 \leq i \leq n$. Note $\eta$ itself is a uniform random permutation on $[n]$. Finally, set $\rho$ to be the product of cycles
\[ \rho := (\eta(a_0+1), \ldots, \eta(a_1)) \cdots (\eta(a_{N_n-1}+1), \ldots, \eta(a_{N_n})). \]
For $1 \leq i \leq n$, let
\[ \tilde{V}_i := \frac{1}{\sqrt{12}} \bigg(V_i - \frac{1}{2}\bigg). \]
Let $1 \leq A_n \leq B_n \leq N_n$ be such that $L_{A_n+1} \geq \cdots \geq L_{B_n}$ are exactly the cycle lengths which are in the interval $[(\log n)^{10}, n / (\log n)^{10}]$. If there are no such cycle lengths, trivially set $A_n := 0, B_n := 0$. The need for introducing $A_n, B_n$ is detailed at two points later on, just before Lemma \ref{lemma:gcm-slope-clt} and just before the proof of Proposition \ref{prop:correlation-joint-asymptotics}. Let $\mc{F}_n := \sigma(N_n, L_1, \ldots, L_{N_n})$. Before continuing, we collect in the following lemma some basic facts about the cycles of random permutations.

\begin{lemma}\label{lemma:cycles-lemma}
For $1 \leq i \leq n$, the expected number of cycles of length $i$ in a uniform random permutation on $[n]$ is exactly $i^{-1}$. Consequently, $\E N_n = O(\log n)$, and also
\[ \E(N_n - (B_n - A_n)) = O(\log \log n). \]
We also have that 
\[ \frac{N_n - \log n}{\sqrt{\log n}} \stackrel{d}{\lra} N(0, 1). \]
Consequently, $N_n / \log n \stackrel{p}{\ra} 1$.
\end{lemma}
\begin{proof}
For the first claim, see e.g. \cite[Theorem 2]{LEN1997}. Using this claim, we have
\[ \E (N_n - (B_n - A_n)) \leq  \sum_{i=1}^{(\log n)^{10}} \frac{1}{i} + \sum_{i=n/(\log n)^{10}}^n \frac{1}{i} = O(\log \log n).\]
For a proof of the central limit theorem, see e.g. \cite[Section 2]{STEELE2002}.
\end{proof}

The following lemma is an intermediate step in simplifying $f_1(\rho), f_2(\rho)$.

\begin{lemma}\label{lemma:introduce-independent-uniform}
We have
\[ \sqrt{n} f_1(\rho) = \frac{1}{\sqrt{n}} \sum_{i=1}^{n-1} (1 - 3|\eta(i)/n - \eta(i+1)/n|) + o_P(1), \]
\[ f_2(\rho) = \sum_{i=A_n+1}^{B_n}  \bigg(\frac{1}{\sqrt{L_i}} \sum_{j=a_{i-1}+1}^{a_i} \tilde{V}_j\bigg)^2 + o_P(\sqrt{\log n}). \]
\end{lemma}
\begin{proof}
First, observe that
\[ \sqrt{n}\bigg( f_1(\rho) - \frac{1}{n} \sum_{C \in \rho} \sum_{(i, j) \in C} (1 - 3|\rho(i)/n - \rho(j)/n|)\bigg) = O\bigg(\frac{1}{n^{3/2}}\bigg).\]
Next, observe
\[ \sqrt{n} \bigg(\frac{1}{n} \sum_{C \in \rho} \sum_{(i, j) \in C} |\rho(i)/n - \rho(j)/n| - \frac{1}{n} \sum_{i=1}^{n-1} |\eta(i)/n - \eta(i+1)/n| \bigg) = O\bigg(\frac{N_n}{\sqrt{n}}\bigg). \]
The first claim now follows, since $\E N_n = O(\log n)$ (by Lemma \ref{lemma:cycles-lemma}).

For the second claim, observe
\[
\bigg| f_1(\rho) - \sum_{i=A_n+1}^{B_n} \bigg(\frac{1}{\sqrt{L_i}} \sum_{j=a_{i-1}+1}^{a_i} \tilde{V}_j\bigg)^2 \bigg| \leq R_1 + R_2,
\]
with 
\[R_1 :=  \sum_{\substack{C \in \rho \\ |C| \notin [(\log n)^{10}, n /(\log n)^{10}]}} \bigg(\frac{1}{\sqrt{|C|}} \sum_{i \in C} x_i\bigg)^2  ,\]
\[ R_2 :=\sum_{i=A_n+1}^{B_n} \bigg|\bigg(\frac{1}{\sqrt{L_i}} \sum_{j=a_{i-1}+1}^{a_i} x_{\eta(j)}\bigg)^2 - \bigg(\frac{1}{\sqrt{L_i}} \sum_{j=a_{i-1}+1}^{a_i} \tilde{V}_j\bigg)^2\bigg|  .\]
To bound $R_1$, observe that if we condition on $\mc{F}_n$, then for any $1 \leq i \leq N_n$, we have
\[ \E\bigg[ \bigg(\frac{1}{\sqrt{L_i}} \sum_{j=a_{i-1+1}}^{a_i} x_{\eta(j)}\bigg)^2 ~\bigg|~ \mc{F}_n\bigg] = O(1). \]
We can thus obtain (recalling Lemma \ref{lemma:cycles-lemma})
\[ \E R_1 = O(\E (N_n - (B_n - A_n))) = O(\log \log n), \]
We thus have $R_1 = o_P(\sqrt{\log n})$.

Next, we will show that $\E R_2 = O(1)$, which implies $R_2 = o_P(\sqrt{\log n})$.
Let $\Delta_n := \sup_{x \in [0, 1]} |\hat{F}_n(x) - x|$. Observe for any $1 \leq i \leq N_n$, we have by Cauchy-Schwarz
\[ \begin{split} \E\bigg[\bigg|\bigg(\frac{1}{\sqrt{L_i}}& \sum_{j=a_{i-1}+1}^{a_i} x_{\eta(j)}\bigg)^2 - \bigg(\frac{1}{\sqrt{L_i}} \sum_{j=a_{i-1}+1}^{a_i} \tilde{V}_j\bigg)^2 \bigg| ~\bigg|~ \mc{F}_n \bigg] \leq S_1 S_2, 
\end{split} \]
where
\[ S_1 := \bigg(\E\bigg[ \bigg(\frac{1}{\sqrt{L_i}}\sum_{j=a_{i-1}+1}^{a_i} x_{\eta(j)} + \tilde{V}_j\bigg)^2 ~\bigg|~ \mc{F}_n\bigg]\bigg)^{1/2} , \]
\[ S_2 := \bigg(\E\bigg[  \bigg(\frac{1}{\sqrt{L_i}} \sum_{j=a_{i-1}+1}^{a_i} x_{\eta(j)} - \tilde{V}_j\bigg)^2 ~\bigg|~ \mc{F}_n\bigg]\bigg)^{1/2}.\]
The inequality $(x+y)^2 \leq 2x^2 + 2y^2$ and a moment computation gives
\[ S_1 = O(1). \]
To bound $S_2$, we have by the definition of $\Delta_n$, the independence of $\Delta_n$ and $\mc{F}_n$, and the facts $|\mu_n - 1/2| \leq 1/n$, $|\sigma_n - 1 / \sqrt{12}| = O(1/n^2)$, $\E \Delta_n^2 = O(1/n)$ (which follows by e.g. the Dvoretzky-Kiefer-Wolfowitz inequality),  
\[ S_2 \leq O(n^{-1/2}) \sqrt{L_i} .\]
We thus obtain
\[ \E R_2 = O(n^{-1/2}) \E \bigg[\sum_{i=A_n+1}^{B_n}\sqrt{L_i}\bigg].  \] 
By Lemma \ref{lemma:cycles-lemma}, we have
\[  \E \bigg[\sum_{i=A_n+1}^{B_n}\sqrt{L_i}\bigg] \leq \sum_{i=1}^n \sqrt{i} \frac{1}{i} = O(n^{1/2}). \]
The desired result now follows.
\end{proof}

We now show that the simplified version of $f_2(\rho)$ given by Lemma \ref{lemma:introduce-independent-uniform} is asymptotically normal. Here is where we crucially use our lower bound on the cycle lengths $L_j$ for $A_n+1 \leq j \leq B_n$, because this ensures that every term in the quantity $T_n$ defined below is approximately a $\chi^2_{(1)}$, meaning $T_n$ is approximately a $\chi^2_{(B_n - A_n)}$.

\begin{lemma}\label{lemma:gcm-slope-clt}
We have
\[ \frac{1}{\sqrt{\log n}}\bigg(\sum_{i=A_n+1}^{B_n}  \bigg(\frac{1}{\sqrt{L_i}} \sum_{j=a_{i-1}+1}^{a_i} \tilde{V}_j\bigg)^2 - \log n\bigg) \stackrel{d}{\lra} N(0, 3).  \]
\end{lemma}
\begin{proof}
Let
\[ T_n := \sum_{i=A_n+1}^{B_n}  \bigg(\frac{1}{\sqrt{L_i}} \sum_{j=a_{i-1}+1}^{a_i} \tilde{V}_j\bigg)^2  .\]
Let $M_n := B_n - A_n$. Since $\E (N_n - M_n) = O(\log \log n)$ (by Lemma \ref{lemma:cycles-lemma}), we have
\[ \frac{1}{\sqrt{\log n}} (T_n - \log n) = \frac{1}{\sqrt{\log n}}(T_n - M_n) + \frac{1}{\sqrt{\log n}} (N_n - \log n) + o_P(1). \]
We know that $(N_n - \log n)/\sqrt{\log n} \stackrel{d}{\ra} N(0, 1)$ (by Lemma \ref{lemma:cycles-lemma}). Thus it suffices to show that for all $\theta \in \R$, we have
\[ \E \bigg[ \exp(i \theta (T_n - M_n) / \sqrt{\log n}) ~\bigg|~ \mc{F}_n\bigg] \stackrel{p}{\lra} \exp(-\theta^2). \]
To start, for $k \geq 1$, let $\phi_k$ be the characteristic function of 
\[ \bigg(\frac{1}{\sqrt{k}} \sum_{i=1}^k \tilde{V}_i\bigg)^2. \]
By the central limit theorem and the continuous mapping theorem, we have that $\phi_k \ra \phi$ pointwise, where $\phi$ is the characteristic function of $\chi^2_{(1)}$. Moreover, we claim that for all $M > 0$, we have
\beq\label{eq:rate-of-conv} \sup_{|\theta| \leq M} |\phi_k(\theta) - \phi(\theta)| = O\bigg(\frac{1 + M^3}{k}\bigg), \eeq
For now let us take the claim as given. Observe
\[\begin{split}
\E \bigg[ \exp(i \theta (T_n - M_n) / \sqrt{\log n}) ~\bigg|~ \mc{F}_n\bigg] =\prod_{j=A_n+1}^{B_n} \phi_{L_j}(\theta / \sqrt{\log n}) e^{-i\theta / \sqrt{\log n}}.
\end{split}\]
Observe moreover that
\[ \begin{split}
\bigg|\prod_{j=A_n+1}^{B_n} \phi_{L_j}(\theta / \sqrt{\log n}) &e^{-i\theta / \sqrt{\log n}} - \prod_{j=A_n+1}^{B_n} \phi(\theta / \sqrt{\log n}) e^{-i\theta / \sqrt{\log n}} \bigg| \leq \\
& \sum_{j= A_n+1}^{B_n} |\phi_{L_j}(\theta / \sqrt{\log n}) - \phi(\theta / \sqrt{\log n})| .
\end{split}\]
By the claim, and the fact that $L_j \geq (\log n)^{10}$ for all $A_n+1 \leq j \leq B_n$, the right hand side above may be bounded by
\[ O\bigg((1 + |\theta|^3)\sum_{j=A_n}^{B_n} \frac{1}{L_j}\bigg) =O\bigg(\frac{(1 + |\theta|^3) N_n}{(\log n)^{10}}\bigg) = o_P(1), \]
where the second equality follows since $\E N_n = O(\log n)$ (by Lemma \ref{lemma:cycles-lemma}). For $k \geq 1$, let $\psi_k$ be the characteristic function of
\[ \frac{1}{\sqrt{k}} \sum_{i=1}^k (A_i - 1), ~~~ A_i \stackrel{i.i.d.}{\sim} \chi^2_{(1)} .\]
Then we have shown that for all $\theta \in \R$,
\[\E \bigg[ \exp(i \theta (T_n - M_n) / \sqrt{\log n}) ~\bigg|~ \mc{F}_n \bigg] - \psi_{M_n} (\theta \sqrt{M_n / \log n}) \stackrel{p}{\lra} 0.   \]
Now observe that $\psi_{M_n}$ is 1-Lipschitz for all $n$, and $M_n  / \log n \stackrel{p}{\ra} 1$ (by Lemma \ref{lemma:cycles-lemma}). Combining this with the central limit theorem, we have that 
\[ \psi_{M_n}(\theta \sqrt{M_n / \log n}) \stackrel{p}{\lra} \exp(-\theta^2). \]
The desired result now follows, modulo the claim \eqref{eq:rate-of-conv}.

To show the claim, first let $W_k := (k^{-1/2} \sum_{i=1}^k \tilde{V}_i)^2$, and let $A \sim \chi^2_{(1)}$. For $\theta \in \R$, let $h_\theta : \R \ra \R$ be defined by $h_\theta(x) := \cos(\theta x)$. Note for $0 \leq i \leq 3$, we have $\sup_{x \in \R} |h^{(i)}(x)| \leq |\theta|^i$. Then by \cite[Theorem 3.1]{GPR2017}, we have that for all $\theta \in \R$, $k \geq 1$,
\[ |\E h_\theta(W_k) - \E h_\theta(A)| = O\bigg(\frac{1 + |\theta|^3}{k}\bigg).\]
Applying this theorem also to the functions $g_\theta(x) := \sin(\theta x)$, we obtain the desired claim.
\end{proof}

The following lemma is the key result needed to show asymptotic independence of $(f_1(\rho), f_2(\rho))$.

\begin{lemma}\label{lemma:perm-osc-reduction}
We have
\[\begin{split}
\frac{1}{\sqrt{n}} \sum_{i=1}^{n-1} (1 - &3|\eta(i)/n - \eta(i+1)/n|)  = \\
&\frac{1}{\sqrt{n}} \sum_{i=1}^{n- n / (\log n)^{9}} (2 - 3|V_i - V_{i+1}| - 6V_i(1-V_i)) + o_P(1). 
\end{split}\]
Moreover, either side above converges in distribution to $N(0, 2/5)$.
\end{lemma}
\begin{proof}
From the proof of Theorem 2 in \cite{A1995} (see in particular equations (5) and (9)), it is shown that 
\[\begin{split}
\frac{1}{n^{1/2} (n-1)}\bigg(\sum_{i=1}^{n-1} &|\eta(i) - \eta(i+1)| - n(n-1)/3\bigg) = \\
&\frac{1}{\sqrt{n}} \sum_{i=1}^{n-1} (|V_i - V_{i+1}| + 2V_i(1-V_i) - 2/3) + o_P(1). 
\end{split}\]
From this it follows that
\[\begin{split}
\frac{1}{\sqrt{n}} \sum_{i=1}^{n-1} (|\eta(i)/n - &\eta(i+1)/n| - 1/3) = \\
&\frac{1}{\sqrt{n}} \sum_{i=1}^{n-1} (|V_i - V_{i+1}| + 2V_i(1-V_i) - 2/3) + o_P(1).
\end{split}\]
By a variance calculation, we have that the right hand side above is equal to
\[\frac{1}{\sqrt{n}} \sum_{i=1}^{n- n / (\log n)^{9}} (|V_i - V_{i+1}| + 2V_i(1-V_i) - 2/3) + o_P(1). \]
The first desired result now follows by combining the previous few observations. The convergence in distribution follows also by the previous few observations and \cite[Theorem 2]{A1995}.
\end{proof}

We are now ready to prove Proposition \ref{prop:correlation-joint-asymptotics}. Since we already have the marginal asymptotics by Lemmas \ref{lemma:gcm-slope-clt} and \ref{lemma:perm-osc-reduction}, the major thing left is to show asymptotic independence of the two marginals. This is where we crucially use our upper bound on the cycle lengths $L_j$ for $A_n+1 \leq j \leq B_n$, to ensure that the two marginals are essentially functions of disjoint sets of the $V_j$ variables.

\begin{proof}[Proof of Proposition \ref{prop:correlation-joint-asymptotics}]
Let
\[ W_n :=  \frac{1}{\sqrt{n}} \sum_{i=1}^{n- n / (\log n)^{9}} (2 - 3|V_i - V_{i+1}| - 6V_i(1-V_i)), \]
\[ Z_n :=  \frac{1}{\sqrt{\log n}}\bigg(\sum_{i=A_n}^{B_n}  \bigg(\frac{1}{\sqrt{L_i}} \sum_{j=a_{i-1}+1}^{a_i} \tilde{V}_j\bigg)^2 - \log n\bigg).\]
By combining \eqref{eq:function-of-permutation}, and Lemmas \ref{lemma:bohnenblust-spitzer-application}, \ref{lemma:introduce-independent-uniform}, \ref{lemma:gcm-slope-clt}, and \ref{lemma:perm-osc-reduction}, we have that it suffices to show the following. Let $f, g : R \ra \R$ be bounded and continuous. Then
\[ \lim_{n \toinf} |\E f(W_n) g(Z_n) - \E f(W_n) \E g(Z_n)| = 0. \]
For each $n$, define the event
\begin{align*} 
E_n &:= \{L_1 + \cdots + L_{A_n} > n - n / (\log n)^{9}\} \\
&= \{L_{A_n+1} + \cdots + L_{N_n} < n / (\log n)^{9}\}
\end{align*}
where we have used the fact $L_1 + \cdots + L_{N_n} = n$. Observe that on this event, the only $\tilde{V}_j$ variables which appear in $Z_n$ must have $j > n - n / (\log n)^{9}$. From this, it follows that
\begin{align*}
\ind_{E_n } \E [f(W_n) g(Z_n) ~|~ \mc{F}_n] &= \ind_{E_n} \E [f(W_n) ~|~ \mc{F}_n] \E[g(Z_n) ~|~ \mc{F}_n] \\
&= \ind_{E_n} \E[f(W_n)] \E [g(Z_n) ~|~ \mc{F}_n],
\end{align*}
where the second identity follows since $V_1, \ldots, V_n$ is independent of $\mc{F}_n$. Letting $C$ be such that $\sup_{x \in \R} |f(x)|, \sup_{x \in \R} |g(x)| \leq C$, we thus have
\[ \limsup_{n \toinf} |\E f(W_n) g(Z_n) - \E f(W_n) \E g(Z_n)| \leq 2C^2 \limsup_{n \toinf}\p(E_n^c). \]
By Lemma \ref{lemma:cycles-lemma}, we have
\[ \E [L_{A_n+1} + \cdots + L_{N_n}] \leq \sum_{i \leq n / (\log n)^{10}} i \frac{1}{i} = \frac{n}{(\log n)^{10}}.  \]
Combining this with Markov's inequality, we obtain
\[ \p(E_n^c) = \p(L_{A_n+1} + \cdots + L_{N_n} \geq n / (\log n)^9) \leq \frac{1}{\log n}. \]
The desired result now follows.
\end{proof}

\subsection{Other features of the isotonic case}

In the isotonic  case there is a simple alternative to $\hat{\chacorr}_{mon}$. Recall the Spearman correlation \cite{SCARSINI1984}, (also known as Spearman's $\rho$) given by, if both $X$ and $Y$ have continuous distribution,
\beq\label{eq:spearman-rho-continuous-def}
\hat{C}_S = \frac{\sum_{i=1}^n (i - \frac{n+1}{2}) S_i}{\sum_{i=1}^n (i - \frac{n+1}{2})^2}
= \frac{1}{n \sigma_n^2} \sum_{i=1}^n \frac{i}{n} \frac{S_i}{n} - \frac{\mu_n}{\sigma_n^2}, 
\eeq
where $S_i = \hat{G}_n(Y_{(i)})$ is the rank of $Y_{(i)}$ (where as usual $Y_{(i)}$ is the $Y$ sample corresponding to the $i$th order statistic $X_{(i)}$), and as in Section \ref{section:moncorr-asymptotics}, we have
\[ \mu_n = \frac{1}{2} \bigg(1 + \frac{1}{n} \bigg), ~~ \sigma_n^2 = \frac{1}{12} \bigg(1 - \frac{1}{n^2}\bigg).\]
For general $(X, Y)$ the population version is
\[ C_S(X, Y) := \mathrm{corr}(F(X), G(Y)) = \frac{\mathrm{Cov}(F(X), G(Y))}{(\Var{F(X)} \Var{G(Y)})^{1/2}}, \]
where $F, G$ are the marginal cdfs of $X, Y$, respectively. It is well known \cite{SCARSINI1984} that $C_S$ satisfies $C_S = 1$ if and only if $Y = g(X)$, where $g$ is strictly increasing, and property B) holds partially, in that $C_S = 0$ if $X, Y$ are independent.

The general estimate of $C_S(X, Y)$ is
\beq \label{eq:spearman-rho-general-def} \hat{C}_S := \frac{\sum_{i=1}^n \hat{F}_n(X_i) \hat{G}_n(Y_i) - \bar{\hat{F}}_n \cdot \bar{\hat{G}}_n}{(\sum_{i=1}^n (\hat{F}_n(X_i) - \bar{\hat{F}}_n)^2 \sum_{i=1}^n (\hat{G}_n(Y_i) - \bar{\hat{G}}_n)^2)^{1/2}} ,\eeq
where $\hat{F}_n, \hat{G}_n$ are the empirical cdfs of $X, Y$, respectively, and
\[ \bar{\hat{F}}_n := \frac{1}{n} \sum_{i=1}^n \hat{F}_n(X_i), \]
and $\bar{\hat{G}}_n$ is similarly defined. In the case $F, G$ are continuous, \eqref{eq:spearman-rho-general-def} reduces to \eqref{eq:spearman-rho-continuous-def}, and further the distribution of $\hat{C}_S$ if $X$ and $Y$ are independent doesn't depend on $F, G$ since $F(X), G(Y) \stackrel{i.i.d.}{\sim} \mathrm{Unif}(0, 1)$. For the case where $F, G$ are continuous (but $X, Y$ are not necessarily independent), the distribution of $\hat{C}_S$ depends only on $g(v | u)$, the conditional density of $V := G(Y)$ given $U := F(X)$, as is the case for Chatterjee's correlation $\hat{\chacorr}_n$.

\subsubsection{A linear expansion}

In the case of general $(X, Y)$ we write
\[ \hat{C}_S = Q(\hat{H}_n), \]
where $\hat{H}_n$ is the empirical joint distribution of the $(X_i, Y_i)$ pairs, $1 \leq i \leq n$, and for any joint distribution $H$, 
\[ Q(H) := \frac{\int F(x) G(y) dH(x, y) - \mu(F) \mu(G)}{\sigma(F) \sigma(G)}, \]
where $F, G$ are the marginals of $H$, and
\[ \mu(F) := \int F(x) dF(x), ~~ \sigma^2(F) := \int F(x)^2 dF(x) - \mu(F)^2. \]
A standard delta method argument yields an expansion for $Q(\hat{H}) - Q(H)$, using
\[\begin{split}
\int \hat{F}_n(x) \hat{G}_n(y) d\hat{H}_n(x, y) = &\int F(x) G(y) dH(x, y) + \int (\hat{F}_n - F)(x) G(y) dH(x, y) \\
&+ \int F(x) (\hat{G}_n - G)(y) dH(x, y) \\
&+\int F(x) G(y) d(\hat{H}_n - H)(x, y) + o_P(n^{-1/2}),
\end{split}\]
and
\[ \mu(\hat{F}_n) = \mu(F) + \int (\hat{F}_n - F)(x) dF(x) + \int F(x) d(\hat{F}_n - F)(x) + o_P(n^{-1/2}), \]
and similar expansions for $\mu(\hat{G}_n)$ as well as $\sigma(\hat{F}_n), \sigma(\hat{G}_n)$. We do not develop the general details here but give the calculation for $F, G$ continuous and $X, Y$ independent. In that case all terms but the last are treated as known, and we can obtain (using integration by parts in the appropriate places)
\beq\label{eq:spearman-rho-iid-sum} \hat{C}_S = \frac{12}{n} \sum_{i=1}^n (U_i - 1/2) (V_i - 1/2) + o_P(n^{-1/2}), \eeq
where $U_i = F(X_i)$, $V_i = G(Y_i)$, so that $U_i, V_i, 1 \leq i \leq n$ are i.i.d. $\mathrm{Unif}(0, 1)$. 

On the other hand, recall
\[ \hat{C}_n(X, Y) = 1 - \frac{3n}{n^2-1} \sum_{i=1}^{n-1} |\hat{G}_n(Y_{(i)}) - \hat{G}_n(Y_{(i+1)})|, \]
where the dependence on the $X$ sample is only through the indices $(i)$. Observe that as in the proof of Lemma \ref{lemma:perm-osc-reduction}, we have from the proof of Theorem 2 in \cite{A1995} (see in particular equations (5) and (9)), we may obtain
\beq\label{eq:chacorr-2-dep-sum} \hat{\chacorr}_n(X, Y) = \frac{3}{n} \sum_{i=1}^{n-1} (2/3 - |V_{(i)} - V_{(i+1)}| - 2V_{(i)} (1-V_{(i)})) + o_P(n^{-1/2}).\eeq
Now by a CLT for triangular arrays of 1-dependent random variables (see e.g. the Theorem of \cite{O1958}), and the Cram\'{e}r-Wold device, we may obtain
\[ (\sqrt{n} \hat{C}_n, \sqrt{n} \hat{C}_S) \stackrel{d}{\lra} N\bigg(\begin{pmatrix} 0 \\ 0 \end{pmatrix}, \begin{pmatrix} \frac{2}{5} & 0 \\ 0 & 1 \end{pmatrix} \bigg), \]
so that in particular, for any $\lambda \in (0, 1)$, we have 
\[ \sqrt{n} (\lambda \hat{C}_n + (1-\lambda) \hat{C}_S) \stackrel{d}{\lra} N(0, 2\lambda^2 / 5 + (1- \lambda)^2 / 12). \]

\subsection{Local power calculations}\label{section:power-calculations}

We will make local power calculations for $\lambda \hat{C}_n + (1-\lambda)\hat{C}^{1/2}_{mon}$, $\lambda \hat{C}_n + (1-\lambda) \hat{C}_S$, and other rank statistics using contiguity theory (see \cite{LCY1990} or \cite[Chapters 6-8]{VDV1998}).

\begin{enumerate}[label=\Roman*]
    \item We begin with a model $\{h_\theta(x, y)\}_{|\theta| < 1}$, where for each $|\theta| < 1$, $h_\theta(x, y)$ is a joint density with respect to Lebesgue measure, and $h_0(x, y) = f_0(x) g_0(y)$ (independence). 
    \item Suppose
    \[ \dot{\ell}_0(x, y) := \frac{\partial}{\partial \theta} \log h_\theta(x, y) \bigg|_{\theta = 0} \]
    exists, and suppose the family $\{h_\theta(x, y)\}_{|\theta| < 1}$ is quadratic mean differentiable at $\theta = 0$ with score function $\dot{\ell}_0$. That is, we have
    \[ \int \int \big(\sqrt{h_\theta(x, y)} - \sqrt{h_0(x, y)} - (1/2) \theta \dot{\ell}_0(x, y) \sqrt{h_0(x, y)}\big)^2 dx dy = o(\theta^2). \]
    Moreover, assume $\E_0 \dot{\ell}_0(X, Y)^2 > 0$ (the assumption of quadratic mean differentiability implies $\E_0[ \dot{\ell}_0(X, Y)^2 ]< \infty$ and $\E_0 \dot{\ell}_0(X, Y) = 0$). 
\end{enumerate}

For $\theta \in \R$ let $P_{n, \theta}$ be the product measure on $([0, 1] \times [0, 1])^n$ with density $\prod_{i=1}^n h_\theta(x, y)$. Take $t \in \R$ and let $\theta_n := t / \sqrt{n}$. By Le Cam's Theorem (see e.g. \cite[Theorem 7.2 and Example 6.5]{VDV1998}), under our assumptions $\{P_{n, \theta_n}\}_{n \geq 1}$ is contiguous with respect to $\{P_{n, 0}\}_{n \geq 1}$. That is if $T_n$ is a function of $((X_i, Y_i), 1 \leq i \leq n)$, then
\[ T_n \stackrel{P_{n, 0}}{\lra} 0 \text{ implies }  T_n \stackrel{P_{n, \theta_n}}{\lra} 0. \]
For our purposes, more importantly, Le Cam's third lemma holds (see e.g. \cite[Theorem 7.2 and Example 6.7]{VDV1998}), stating if 
\beq\label{eq:le-cam-3} \bigg(\sqrt{n}T_n, \frac{1}{\sqrt{n}} \sum_{i=1}^n \dot{\ell}_0(X_i, Y_i)\bigg) \stackrel{d}{\lra} N\bigg(\begin{pmatrix} 0 \\ 0 \end{pmatrix}, \begin{pmatrix} \sigma^2 & \rho \sigma \tau_0 \\ \rho \sigma \tau_0 & \tau_0^2 \end{pmatrix}\bigg) \eeq
under $\{P_{n, 0}\}_{n \geq 1}$, where $\tau_0^2 := \E_0 [\dot{\ell}_0(X, Y)^2]$ is the Fisher information, then under $\{P_{n, \theta_n}\}_{n \geq 1}$, we have
\[ \sqrt{n} T_n \stackrel{d}{\lra} N(tc, \sigma^2), \]
where $c := \rho \sigma \tau_0$ is the asymptotic covariance (under $\{P_{n, 0}\}_{n \geq 1}$) between the statistic $T_n$ and the score statistic. 

Assuming \eqref{eq:le-cam-3} holds under $\{P_{n, 0}\}_{n \geq 1}$, $\sqrt{n} T_n / \sigma$ can be viewed as a test statistic for the hypothesis of independence, while
\[ L_n := \frac{1}{\tau_0 \sqrt{n}} \sum_{i=1}^n \dot{\ell}_0(X_i, Y_i)\]
can be viewed as the asymptotically optimal test statistic for the family $\{h_\theta(x, y)\}_{|\theta| < 1}$. The Pitman efficiency (see e.g. \cite[Chapter 8]{VDV1998}) of the first test to the second is by \eqref{eq:le-cam-3}, 
\[ e(T_n) = \rho^2. \]
Note that $\rho(L_n) = 1$. 

\begin{prop}\label{prop:chacorr-zero-power}
Suppose the listed assumptions I and II hold. Then we have $e(\hat{C}_n) = 0$.
\end{prop}
\begin{remark}
In our initial draft, we had initially proven this result for some specific models. After we became aware of the related work of Shi et al. \cite{SDH2020}, we realized that the result held in the present general setting. As such, the following proof is a small adaptation of an argument in \cite{SDH2020}. In particular, we utilize their very nice and crucial observation that some fortuitous cancellation occurs. 
\end{remark}
\begin{proof}
Throughout, all expectations and covariances will be under $\theta = 0$, so we omit writing $\E_0$ and $\mathrm{Cov}_0$. Recalling \eqref{eq:chacorr-2-dep-sum}, we have that under $\{P_{n, 0}\}_{n \geq 1}$, 
\[\hat{C}_n = \frac{1}{n} \sum_{i=1}^n (2/3 - |V_{(i)} - V_{(i+1)}| - 2 V_{(i)} (1 - V_{(i)})) + o_P(n^{-1/2}), \]
where $V_i := G_0(Y_i)$, $G_0$ is the cdf of $Y$ under $\theta = 0$, and $(n+1) := (1)$. It thus suffices to show that for all $n \geq 2$,
\[  \mathrm{Cov}\bigg(\sum_{i=1}^n |V_{(i)} - V_{(i+1)}|, \sum_{i=1}^n \dot{\ell}_0(X_i, Y_i)\bigg) = - 2\mathrm{Cov}\bigg(\sum_{i=1}^n V_{(i)} (1 - V_{(i)}), \sum_{i=1}^n \dot{\ell}_0(X_i, Y_i)\bigg) .\]
We have
\begin{align*}
\mathrm{Cov}\bigg(\sum_{i=1}^n V_{(i)} (1 - V_{(i)}), \sum_{i=1}^n \dot{\ell}_0(X_i, Y_i)\bigg) &= \mathrm{Cov}\bigg( \sum_{i=1}^n V_i(1-V_i), \sum_{i=1}^n \dot{\ell}_0(X_i, Y_i)\bigg)\\
&= \sum_{i=1}^n \mathrm{Cov}(V_i(1-V_i), \dot{\ell}_0(X_i, Y_i)) \\
&= \sum_{i=1}^n \E[ V_i(1- V_i) \dot{\ell}_0(X_i, Y_i)]. 
\end{align*}
where the second to last equality follows since $\E \dot{\ell}_0(X, Y) = 0$. For the other covariance term, let $\pi$ be the permutation defined by if $i = (j)$, then $\pi(i) := (j+1)$. Then
\[ \sum_{i=1}^n |V_{(i)} - V_{(i+1)}| = \sum_{i=1}^n |V_i - V_{\pi(i)}|. \]
Thus
\begin{align*}
\mathrm{Cov}\bigg(\sum_{i=1}^n |V_{(i)} - V_{(i+1)}|, \sum_{i=1}^n \dot{\ell}_0(X_i, Y_i)\bigg) &= \sum_{i, j=1}^n \mathrm{Cov}(|V_{i} - V_{\pi(i)}|, \dot{\ell}_0(X_j, Y_j)) \\
&= \sum_{i,j=1}^n \E [|V_{i} - V_{\pi(i)}| \dot{\ell}_0(X_j, Y_j)].
\end{align*}
In the case $i \neq j$, we have
\[\begin{split}
\E [|V_{i} - V_{\pi(i)}| \dot{\ell}_0(X_j, Y_j)] = \E\big[ &\ind(\pi(i) \neq j) \E[|V_i - V_{\pi(i)}| \dot{\ell}_0(X_j, Y_j) ~|~ X_1, \ldots, X_n]\big] \\
&+ \E[ \ind(\pi(i) = j) |V_i - V_j| \dot{\ell}_0(X_j, Y_j)] .
\end{split}\]
On the event $\pi(i) \neq j$, we have that $Y_i, Y_{\pi(i)}, Y_j$ are all conditionally independent given $X_1, \ldots, X_n$, and thus 
\[\begin{split}
\ind(\pi(i) \neq j) \E[|V_i - V_{\pi(i)}| &\dot{\ell}_0(X_j, Y_j) ~|~ X_1, \ldots, X_n] =\\
&\ind(\pi(i) \neq j) \E [|V_i - V_{\pi(i)}|] \E[\dot{\ell}_0(X_j, Y_j) ~|~ X_1, \ldots, X_n].
\end{split}\]
Note $V_i, V_{\pi(i)} \stackrel{i.i.d.}{\sim} \mathrm{Unif}(0, 1)$, and so $\E [|V_i - V_{\pi(i)}|]$ is constant in $i$. Note also
\[ \sum_{\substack{i, j=1 \\ i \neq j}}^n \E[ \ind(\pi(i) \neq j) \dot{\ell}_0(X_j, Y_j)] = (n-2)\sum_{j=1}^n \E \dot{\ell}_0(X_j, Y_j) = 0.\]
Thus upon combining the previous few displays, we obtain
\[ \begin{split}
\sum_{i,j=1}^n \E [|V_{i} - V_{\pi(i)}| &\dot{\ell}_0(X_j, Y_j)] = \sum_{i=1}^n \E [|V_i - V_{\pi(i)}|  \dot{\ell}_0(X_i, Y_i)] ~+ \\
&\sum_{\substack{i,j=1 \\ i \neq j}}^n \E [\ind(\pi(i) = j) |V_i - V_j| \dot{\ell}_0(X_j, Y_j) ]. 
\end{split}\]
The right hand side above may be more simply written
\[ \sum_{i=1}^n \E [|V_i - V_{\pi(i)}| \dot{\ell}_0(X_i, Y_i)] + \sum_{i=1}^n \E [|V_{\pi^{-1}(i)} - V_i| \dot{\ell}_0(X_i, Y_i)].\]
To finish, we want to show that the above is equal to
\[ -2 \E \sum_{i=1}^n \E[ V_i(1- V_i) \dot{\ell}_0(X_i, Y_i)] = -2 n\E [V(1-V) \dot{\ell}_0(X, Y)].\]
Towards this end, observe for any $i$, we have $Y_i, Y_{\pi(i)}, X_i$ are independent, with $Y_i, Y_{\pi(i)} \stackrel{i.i.d.}{\sim} G_0$, and thus we obtain
\[ \E [|V_i - V_{\pi(i)}| \dot{\ell}_0(X_i, Y_i)] = \E [|V - V'| \dot{\ell}_0(X, Y)], \]
where $V = G_0(Y)$, and $V' \stackrel{d}{=} V$ is independent of $X, Y$. We similarly have
\[ \E [|V_{\pi^{-1}(i)} - V_i| \dot{\ell}_0(X_i, Y_i)] = \E [|V - V'| \dot{\ell}_0(X, Y)]. \]
Thus 
\[ \sum_{i=1}^n \E [|V_i - V_{\pi(i)}| \dot{\ell}_0(X_i, Y_i)] + \sum_{i=1}^n \E [|V_{\pi^{-1}(i)} - V_i| \dot{\ell}_0(X_i, Y_i)]  = 2n \E [|V - V'| \dot{\ell}_0(X, Y)].\]
To finish, observe that 
\[ \E[|V - V'| ~|~ X, Y] = \E[|V - V'| ~|~ V] = \frac{1}{2}(V^2 + (1 - V)^2), \]
so that
\[ \E [|V - V'| \dot{\ell}_0(X, Y)] = \frac{1}{2} \E [(2V^2 - 2V + 1) \dot{\ell}_0(X, Y)] = - \E [V(1-V) \dot{\ell}_0(X, Y)],\]
where the second equality follows since $\E \dot{\ell}_0(X, Y) = 0$. The desired result now follows.
\end{proof}

\begin{remark}
In addition to the models considered by Shi et al. \cite{SDH2020}, we note another class of models to which Proposition \ref{prop:chacorr-zero-power} applies, similar to the trend model considered by Chatterjee \cite{CH2019}.

Here
\[ Y = \theta a(X) + \varep, \]
where $X$ and $\varep$ are independent with densities $f_0, g_0$, respectively. The joint density is
\[ h_\theta(x, y) = f_0(x) g_0(y - \theta a(x)). \]
If $g_0$ is differentiable, and
\[ I(g) := \int_{-\infty}^\infty \frac{g_0'(y)^2}{g_0(y)} dy < \infty, \]
then it is easy to see that the family $\{h_\theta(x, y) : |\theta| < 1\}$ is quadratic mean differentiable at $\theta = 0$, with 
\[ \dot{\ell}_0(X, Y) = -a(X) \frac{g_0'(Y)}{g_0(Y)}. \]
Hence $e(\hat{C}_n) = 0$ here too. In fact, this is true even if we make the model semiparametric with $g_0$ unknown.
\end{remark}

\subsubsection{Comparison of $\hat{C}_n$ to $\hat{\tilde{C}}_{mon}$}

Recall that we defined
\[ \hat{\tilde{C}}_{mon} := (1/2) (\hat{C}_n + \hat{C}^{1/2}_{mon}).\]
Let us generalize this slightly to
\[ \hat{\tilde{C}}_{mon, \lambda} := \lambda \hat{C}_n + (1-\lambda) \hat{C}^{1/2}_{mon}. \]
As in the proof of Theorem \ref{thm:averaged-correlation-asymptotics}, we can obtain that for $X, Y$ independent and continuously distributed,
\[ \sqrt{n} (\hat{\tilde{C}}_{mon, \lambda} - (1-\lambda) \sqrt{\log n / n}) \stackrel{d}{\lra} \frac{2}{5} \lambda^2 + \frac{3}{4} (1-\lambda)^2.\]
Now note that $\hat{C}^{1/2}_{mon}$ and $L_n$ are asymptotically independent. This follows by essentially the same argument given in Section \ref{section:moncorr-asymptotics} for the independence of $\hat{C}_n$ and $\hat{C}_{mon}$. Therefore, $\hat{\tilde{C}}_{mon, \lambda}$ always has Pitman efficiency less than that of $\hat{C}_n$ unless $\lambda = 1$. Specifically, the asymptotic correlation betwewen $L_n$ and $\hat{\tilde{C}}_{mon, \lambda}$ is
\[ \frac{\lambda}{(\frac{2}{5} \lambda^2 + \frac{3}{4} (1-\lambda)^2)^{1/2}} \lim_{n \toinf} \mathrm{Cov}(\hat{C}_n, L_n). \]
This is equal to
\[ \frac{1}{(1 + \frac{15}{8} (\frac{1-\lambda}{\lambda})^2)^{1/2}} \lim_{n \toinf} \frac{\mathrm{Cov}(\hat{C}_n, L_n)}{\sqrt{2/5}} = \frac{1}{(1 + \frac{15}{8} (\frac{1-\lambda}{\lambda})^2)^{1/2}} e^{1/2}(\hat{C}_n).\]
We thus obtain
\[ e(\hat{\tilde{C}}_{mon, \lambda}) = \frac{e(\hat{C}_n)}{1 + \frac{15}{8} (\frac{1-\lambda}{\lambda})^2}. \]
When $\lambda = 1/2$, we have
\[ e(\hat{\tilde{C}}_{mon}) = \frac{8}{23} e(\hat{C}_n). \]

\subsection{A copula generated model}\label{section:semiparametric}

We believe the appropriate setting for rank statistic power calculations is a copula generated model as defined next. We start with a parametric model of densities $\{h_\theta : |\theta| < 1\}$, where $h_\theta : (0, 1)^2 \ra [0, \infty)$ for all $|\theta| < 1$,  and $h_0(x, y) \equiv 1$. Let
\[ \mc{F} := \{a : [0, 1] \ra [0, 1] \text{ absolutely continuous}, a' > 0, a(0) = 0, a(1) = 1 \}.\]
I.e., $\mc{F}$ is the set of all absolutely continuous, strictly increasing transformations which map $0$ to $0$ and $1$ to $1$. Given $q, r \in \mc{F}$, we may define
\[ h_\theta(x, y, q, r) := h_\theta(q(x), r(y)) q'(x) r'(y), ~~ (x, y) \in (0, 1)^2. \]
Note that if $h_\theta$ is the density of $(X, Y)$, then $h_\theta(\cdot, \cdot, q, r)$ is the density of the pair $(q^{-1}(X), r^{-1}(Y))$. Define a semiparametric model by
\[ \mc{P} := \{h_\theta(\cdot, \cdot, q, r) : |\theta| < 1, q, r \in \mc{F} \}.\]
Note that $h_0(x, y, q, r) = q'(x) r'(y)$ corresponds to independence for all $(q, r)$ pairs.

Let 
\[ \mc{Q} := \{Q_\theta : (F_\theta(X), G_\theta(Y)) \sim Q_\theta, ~(X, Y) \sim h_\theta, |\theta| < 1\}.\]
Then $\mc{Q}$ is a copula which generates the same semiparametric model as $\{h_\theta : |\theta| < 1\}$.

Rank test statistics have the property that their distribution depends only on $\theta$ and not on $q$ or $r$. Therefore, we believe calculations should be carried out for parametric submodels where $q$ and $r$ also depend on $\theta$.

Consider a parametric submodel, with $\bar{\theta} = (\theta_1, \theta_2, \theta_3)$, and
\[ \dens_{\bar{\theta}}(x, y) := h_{\theta_1}(x, y, q_{\theta_2}, r_{\theta_3}), ~~ |\bar{\theta}|_\infty < 1. \]
Let $\ell_{\bar{\theta}} := \log \dens_{\bar{\theta}}$, and let
\[\nabla \ell_{\bar{0}} (x, y) := (D_{1}(x, y), D_{2}(x, y), D_{3}(x, y)) \]
be the gradient of $\ell$ (in the $\bar{\theta}$ variable) at $\bar{\theta} = 0$. We suppose that the family $\{\dens_{\bar{\theta}}: |\bar{\theta}|_\infty < 1\}$ is quadratic mean differentiable at $\bar{\theta} = \bar{0}$. That is, it satisfies
\beq\label{eq:QMD}\tag{QMD} \E_{\bar{0}} \big(\dens_{\bar{\theta}}(X, Y)^{1/2} - \dens_{\bar{0}}(X, Y)^{1/2}- (1/2) \dens_{\bar{0}}(X, Y)^{1/2} (\nabla \ell_{\bar{0}}(X, Y), \bar{\theta})\big)^2 = o(|\bar{\theta}|_2^2),\eeq
where $\E_{\bar{0}}$ signifies that the random variables $(X, Y)$ are distributed according to $\dens_{\bar{0}}(x, y) = q_0'(x) r_0'(y)$. Note that
\[ D_{1}(X, Y) = \frac{\partial}{\partial \theta_1} \log h_{\theta_1}(q_{\theta_2}(X), r_{\theta_3}(Y)) \bigg|_{\bar{\theta}= 0} = \dot{\ell}_0(q_0(X), r_0(Y)), \]
(here $\dot{\ell}_0 := \ptl \log h_\theta / \ptl \theta |_{\theta = 0}$) and since $h_{\bar{0}} \equiv 1$, we have
\[ D_{2} (X, Y) = \frac{\ptl}{\ptl \theta_2} \log q_{\theta_2}'(X) \bigg|_{\theta_2 = 0}, ~~~ D_{3}(X, Y) = \frac{\ptl}{\ptl \theta_3} \log r_{\theta_3}'(Y)  \bigg|_{\theta_3 = 0}.\]

\begin{remark}
Let 
\beq\label{eq:q-theta-r-theta-def}
\begin{aligned}
q_\theta'(x) &:= q_0'(x) \big(1 + \theta a(q_0(x))\big), \\
r_\theta'(y) &:= r_0'(y) \big(1 + \theta b(r_0(y))\big), 
\end{aligned}
\eeq
where $|a| < 1, |b| < 1$, and $\E a(U) = \E b(U) = 0$, for $U \sim \mathrm{Unif}(0, 1)$. If $\{h_\theta : |\theta| < 1\}$ satisfies assumptions I and II, then the parametric copula model defined by \eqref{eq:q-theta-r-theta-def} above satisfies \eqref{eq:QMD}.
\end{remark}

The tangent spaces (using the notation of \cite{BKRW1993}) of the model defined by \eqref{eq:q-theta-r-theta-def} at $\theta_1 = \theta_2 = \theta_3 = 0$ are
\[ \dot{\mc{P}}_{\theta_1} = [\dot{\ell}_0(q_0(X), r_0(Y))],~~ \dot{\mc{P}}_{\theta_2} = [a(q_0(X))], ~~ \dot{\mc{P}}_{\theta_3} = [b(r_0(Y))]. \]
Since the set of $a, b$ as above are dense in $L_2^0([0, 1]) := \{c(U) : \E [c(U)^2] < \infty, \E c(U) = 0\}$, the tangent spaces of the semiparametric copula model are
\[ \dot{\mc{P}}_{\theta} = [\dot{\ell}_0(q_0(X), r_0(Y))], \]
\[\dot{\mc{P}}_q = \{f(q_0(X)) : f \in L_2^0([0, 1])\}, ~~ \dot{\mc{P}}_r = \{f(r_0(Y)) : f \in L_2^0([0, 1])\}.\]

While $\dot{\ell}_0(q_0(X), r_0(Y))$ is the score function for the base model $\{h_\theta(\cdot, \cdot, q_0, r_0) : |\theta| < 1\}$, we now obtain that
\[ \dot{\ell}_0^*(X, Y) := \dot{\ell}_0(q_0(X), r_0(Y)) - \E[\dot{\ell}_0(q_0(X), r_0(Y)) ~|~ X] - \E[\dot{\ell}_0(q_0(X), r_0(Y)) ~|~ Y] \]
is the efficient score function for the model $\mc{P}$, since $\E[\dot{\ell}_0(q_0(X), r_0(Y)) ~|~ X]$ is the projection of $\dot{\ell}_0(q_0(X), r_0(Y))$ on $\dot{\mc{P}}_q$ and similarly for $\dot{\mc{P}}_r$ (see \cite[Chapter 3]{BKRW1993}). 

Consider the submodel of $\mc{P}$, given by \eqref{eq:q-theta-r-theta-def}, with $\theta_1 = \theta_2 = \theta_3 = \theta$, and
\[ a(q_0(X)) = -\E [\dot{\ell}_0(q_0(X), r_0(Y)) ~|~ X], ~~ b(r_0(Y)) = -\E[\dot{\ell}_0(q_0(X), r_0(Y)) ~|~ Y]. \]
Given assumptions I and II, we have that \eqref{eq:QMD} holds, and hence this model is also quadratic mean differentiable at $\theta = 0$, with score function $\dot{\ell}_0^*$. For estimation, this means that if an estimate is regular and linear on $\mc{P}$, then its asymptotic variance $\sigma^2 / n \geq (I^*)^{-1} / n$, where
\[ I^* := \E_0[\dot{\ell}_0^*(q_0(X), r_0(Y))^2]. \]
For a rank test statistic $T_n$ such that under $\theta = 0$, we have that
\beq\label{eq:semiparametric-joint-clt} \bigg(\sqrt{n}T_n, \frac{1}{\sqrt{n}} \sum_{i=1}^n  \dot{\ell}_0^*(X_i, Y_i) \bigg) \stackrel{d}{\lra} N\bigg(\begin{pmatrix} 0\\ 0 \end{pmatrix},  \begin{pmatrix} \sigma^2 & \rho^* \sigma \sqrt{I^*} \\ \rho^* \sigma \sqrt{I^*} & I^* \end{pmatrix}\bigg), \eeq
then as in Section \ref{section:power-calculations}, the Pitman efficiency of $T_n$ with respect to the optimal asymptotic test is $e^*(T_n) = (\rho^*)^2$. 

\begin{remark}
We can exhibit explicit families $\{q_\theta : |\theta| < 1\}$, $\{r_\theta : |\theta| < 1\}$ such that the parametric model $\{h_\theta(\cdot, \cdot, q_\theta, r_\theta) : |\theta| < 1\}$ has score function $\dot{\ell}_0^*$. Set $q_\theta := F_\theta^{-1}$, $r_\theta := G_\theta^{-1}$, where $F_\theta, G_\theta$ are the respective marginal cdfs of $X, Y$ under $h_\theta$. (Note then the distribution $h_\theta(\cdot, \cdot, q_\theta, r_\theta)$ is given by the law of $(F_\theta(X), G_\theta(Y))$, if $(X, Y)$ is distributed according to $h_\theta$. So we are essentially going from the semiparametric copula generated model back to the parametric copula model.) Let $\dot{s}_0$ denote the score function at $\theta = 0$ for this model. A change of variables calculation then gives that 
\[ \dot{s}_0(X, Y) = \dot{\ell}_0(q_0(X), r_0(Y)) - \frac{1}{f_0(X)} \frac{\ptl f_\theta}{\ptl \theta} (q_0(X)) \bigg|_{\theta = 0} - \frac{1}{g_0(Y)} \frac{\ptl g_\theta}{\ptl \theta}(r_0(Y)) \bigg|_{\theta = 0}. \]
It remains to show that under $\theta = 0$, we have
\[  \E[\dot{\ell}_0(q_0(X), r_0(Y)) ~|~ X]  = \frac{1}{f_0(X)} \frac{\ptl f_\theta}{\ptl \theta} (q_0(X)) \bigg|_{\theta = 0}, \]
and similarly for the other term. Since $h_0 \equiv 1$, and thus $f_0 \equiv 1$, this reduces to showing
\[ \E\bigg[ \frac{\ptl h_\theta}{\ptl \theta} (q_0(X), r_0(Y)) \bigg|_{\theta = 0} ~\bigg|~ X \bigg] = \frac{\ptl f_\theta}{\ptl \theta} (q_0(X)) \bigg|_{\theta = 0}.\]
Assuming $h_\theta$ has sufficient smoothness (e.g., there is some $\varep > 0$ and $H(q_0(X), r_0(Y))$ such that $\sup_{|\theta| < \varep} |(\ptl h_\theta / \ptl \theta) (q_0(X), r_0(Y))| \leq H(q_0(X), r_0(Y))$, and $\E H(q_0(X), r_0(Y)) < \infty$) we can interchange the conditional expectation and the differentiation, to obtain that the left hand side is
\[ \frac{\ptl}{\ptl \theta} \E[ h_\theta(q_0(X), r_0(Y)) ~|~ X] \bigg|_{\theta = 0}. \]
To finish, we have (note at $\theta = 0$, we have that $Y$ is independent of $X$, and has density $r_0'$)
\[ \E[h_\theta(q_0(X), r_0(Y)) ~|~ X]  = \int h_\theta(q_0(X), r_0(y)) r_0'(y)dy = f_\theta(q_0(X)).  \]
\end{remark}

Note for any rank test statistic $T_n$ satisfying \eqref{eq:semiparametric-joint-clt}, the model with score function $\dot{\ell}_0^*$ at $\theta = 0$ is least favorable in terms of power, and $(\rho^*)^2 \leq \rho^2$, where $\rho$ is as in Section \ref{section:power-calculations}. For $\hat{C}_n$, we have established that $\rho^2 = 0$. 

For the semiparametric copula generated model, we can get a result which is weaker for $\hat{C}_n$, but more general in a way we point out. 

\begin{prop}
Suppose we have a rank statistic $T_n$ which satisfies \eqref{eq:semiparametric-joint-clt}, and such that for every $n$, there exists functions $q_{x, n}, q_{y, n}$, such that 
\[ \sqrt{n} T_n = q_{x, n}(X_1, \ldots, X_n) + q_{y, n}(Y_1, \ldots, Y_n) + o_P(1) \]
Then under the submodel of $\mc{P}$ which has score function $\dot{\ell}_0^*$ at $\theta = 0$, we have that $e^*(T_n) = 0$.
\end{prop}
\begin{proof}
Note
\begin{align*}
\mathrm{Cov}_0\bigg(q_{x, n}(X_1, \ldots, X_n) , \sum_{i=1}^n \ell_0^*(X_i, Y_i)\bigg) &= \sum_{i=1}^n \E_0 [q_{x, n}(X_1, \ldots, X_n) \ell_0^*(X_i, Y_i)] \\
&= \sum_{i=1}^n \E_0[ q_{x, n}(X_1, \ldots, X_n) \E_0[\ell_0^*(X_i, Y_i) ~|~ X_i] ] \\
&= 0,
\end{align*}
where the final equality follows since $\E_0[\ell_0^*(X_i, Y_i) ~|~ X_i] = 0$. The other term is handled similarly.
\end{proof}

\begin{remark}
Dette et al. \cite{DSS2013} defined a statistic $\hat{r}_n$ (see \cite[equation (13)]{DSS2013}), which under the right assumptions, is consistent for the correlation $C$ defined by \eqref{eq:chacorr-def}, and moreover satisfies a central limit theorem. If additionally, $(X, Y)$ are independent, then (see in particular equation (24) in the proof of \cite[Theorem 3]{DSS2013})
\[ \frac{1}{6} \hat{r}_n + \frac{1}{3} = T_0(Y_1, \ldots, Y_n) + T_1(X_1, \ldots, X_n) + T_2(Y_1, \ldots, Y_n) + o_P(n^{-1/2}),\]
for some functions $T_0, T_1, T_2$. Thus the previous proposition implies that $e(\hat{r}_n) = 0$ for quadratic mean differentiable families which have score function $\dot{\ell}_0^*$. On the other hand, see Shi et al. \cite{SDH2020} for a family for which $\hat{r}_n$ is rate optimal. 

We conjecture that the poor local behavior of $\hat{C}_n$ and $\hat{r}_n$ compared with other tests of independence which are consistent against all alternatives, as shown by Shi et al. \cite{SDH2020}, is due to their property of being exactly 1 at the population level if and only if $Y = g(X)$ a.s. for some $g$, which is not shared by any of the competitors. We believe that if one adds smoothness assumptions to $H : Y = g(X)$ a.s., then that gives these statistics uniquely good local power against alternatives to this hypothesis. 
\end{remark}

\section*{Acknowledgements}
We thank Sourav Chatterjee for facilitating this collaboration, as well as for helpful conversations. We thank Holger Dette for pointing out an important reference. We thank Hongjian Shi, Fang Han, and Mathias Drton for valuable comments on the local power calculations, which led to an improvement of our results.

\end{document}